\documentclass[reqno, 11pt,oneside]{amsart}
\usepackage{amsmath,amsthm}
\usepackage{color,enumerate}
\usepackage{amssymb}
\usepackage{upgreek}
\usepackage{mathrsfs}
\usepackage{mathabx}
\usepackage{txfonts}
\usepackage{graphicx}
\usepackage{enumitem}		% special itemize with custom tags and labels that work 
\usepackage{soul}
\usepackage[colorlinks=true,linkcolor=blue, urlcolor=red, citecolor=blue,backref]{hyperref}	

\definecolor{teal}{rgb}{0.0, 0.5, 0.5}

\usepackage{soul}

\usepackage[normalem]{ulem}
\definecolor{forestgreen(web)}{rgb}{0.13, 0.55, 0.13}

\newcommand{\vep}{\varepsilon}

\theoremstyle{plain}
\newtheorem{maintheorem}{Theorem}

\newtheorem{maincorollary}[maintheorem]{Corollary}

\newtheorem{proposition}{Proposition}[section]
\newtheorem{lemma}{Lemma}[section]

\newtheorem{claim}{Claim}[section]
\newtheorem{definition}{Definition}[section]
\newtheorem{remark}{Remark}[section]
\newtheorem{example}{Example}[section]

\def\R{\mathbb{R}}

\def\Z{\mathbb{Z}}

\definecolor{bgreen}{rgb}{0.13, 0.55, 0.13}

\definecolor{forestgreenweb}{RGB}{34,139,34}
\definecolor{darkorange}{RGB}{255,140,0}

\begin{document}

%\title
%[Topological and ergodic features of $C^0$-Baire generic impulsive semiflows]{Topological and ergodic features of $C^0$-Baire \\ generic impulsive semiflows}
\title
[On the periodic orbits of $C^0$-typical impulsive semiflows]{On the periodic orbits of $C^0$-typical \\ impulsive semiflows}

\author[M. Bessa]{M. Bessa}
\address{M\'ario Bessa\\ CEMS.UL and Departamento de Matematica da Univesidade Aberta, Portugal}
\email{mario.costa@uab.pt}

 \author[J. Siqueira]{J. Siqueira}
\address{Jaqueline Siqueira\\ Departamento de Matem\'atica, Instituto de Matem\'atica, Universidade Federal do Rio de Janeiro, Caixa Postal 68530, Rio de Janeiro, RJ 21945-970, Brazil}
\email{jaqueline@im.ufrj.br}

\author[M. J. Torres]{M. J. Torres}
\address{Maria Joana Torres, CMAT and Departamento de Matem\'atica, Universidade do Minho, 
Campus de Gualtar,
4700-057 Braga, Portugal}
\email{jtorres@math.uminho.pt}

%
%Paulo Varandas \\ Departamento de Matem\'atica
%da Universidade de Aveiro \\
%Campus Universitário de Santiago \\ 3810-193 Aveiro, Portugal  \\
%E-mail: paulo.varandas@ua.pt

\author[P. Varandas]{P. Varandas}
\address{
Paulo Varandas, CIDMA and Departamento de Matem\'atica
da Universidade de Aveiro, Campus Universitário de Santiago, 3810-193 Aveiro, Portugal }
\email{paulo.varandas@ua.pt}

\date{\today}


\begin{abstract}
Impulsive dynamical systems, modeled by a continuous 
semiflow and an impulse function, may be discontinuous
and may have non-intuitive topological properties, as the non-invariance of the non-wandering set or the non-existence of invariant probability measures.
In this paper we study dynamical features of impulsive flows parameterized by the space of impulses. 
We prove that impulsive semiflows determined by a $C^1$-Baire generic impulse are such that the set of hyperbolic periodic orbits is dense in the 
set of non-wandering points which meet the impulsive region.  As a consequence, we provide sufficient conditions for the non-wandering set of a typical impulsive semiflow (except the discontinuity set) to be invariant. 
Several applications are given concerning impulsive semiflows obtained from billiard, Anosov and geometric Lorenz flows.  
\end{abstract}

\begin{abstract}
Impulsive semiflows modeled by 
{continuous flows} and continuous impulsive functions, defined over an impulsive region, are piecewise continuous semiflows 
with piecewise smooth trajectories.  
In this paper we contribute to the topological description of typical impulsive semiflows, parameterized by both vector fields and impulses.
We prove that
 $C^0$-generic {continuous flows} generate impulsive semiflows with denseness of periodic orbits on the 
 non-wandering set.
 Additionally, we show that
$C^0$-generic impulses generate impulsive semiflows with denseness of periodic orbits on the impulsive non-wandering set.  
\end{abstract}

\maketitle

%\tableofcontents

%%%%%%%%%%%%%%%%%%%%%%%%%%%%%%%%%%%%%%

%%%%%%%%%%%%%%%%%%%%%%%%%%%%%%%%%%%%%%
\section{Introduction}

Classical dynamical systems theory has proven essential in explaining a broad array of phenomena. 
However, real-world systems often experience sudden, discontinuous shifts, a factor that traditional dynamical systems theory tends to overlook. 
The theory of impulsive semiflows models such abrupt changes on the dynamics caused by external input and provide a quite rich playground to the study of the geometric, ergodic and topological properties of continuous-time dynamical systems under perturbations of the system by a discrete-time impulsive map. 
These systems, characterized by a continuous semiflow acting on a compact metric space $(M, d)$, a subset $D \subset M$ where the semiflow experiences abrupt perturbations, and an impulsive function $I: D \to I(D)\subset M$ that models these perturbations, frequently exhibit discontinuities, adding a layer of complexity to their behavior.
The study of the topological properties of impulsive systems has been extensively studied in the last two decades (see e.g.  \cite{BBCC} and references therein).

\medskip
In analogy to classical dynamical systems, impulsive semiflows are likely to behave quite differently depending on the regularity of the unperturbed flows and the space of impulses. In \cite{STV1}, we initiated the study of the geometric and topological features for $C^1$-typical impulsive flows by proving that, given an initial smooth flow, there exists a $C^1$-Baire generic subset of impulses (in the space of $C^1$ diffeomorphisms of $D$ onto its image)
such that the corresponding impulsive semiflow admits a dense set of periodic orbits on the part of the non-wandering set that intersects the impulsive region. 
The methods, relying on the extension of the concepts of hyperbolicity and connecting lemmas to $C^1$-impulsive semiflows (possibly with discontinuities), depend strongly on the $C^1$-topology in the space of impulses. 

\medskip
In this paper we consider impulsive semiflows  generated by {continuous flows} and impulses, chosen as homeomorphisms. 
%{\color{blue}{onto its image}}.
%{\color{red}{In this way, the space of impulsive semiflows is parameterized by  a {flow $\varphi\in\mathscr{F}(M)$} and an impulse $I \in \mathscr{I}_{D,\hat D}$ (cf. Section~\ref{Impulsivesemi} for details). Talvez retirar.}}
 Endowing the space of {flows} and impulses with the $C^0$-topology, a first natural question is to ask whether a typical impulsive semiflow (usually piecewise continuous) admits invariant measures or periodic points. 
The denseness of periodic orbits in the non-wandering set, often known as general density theorem, was obtained first by Pugh for $C^1$-generic diffeomorphisms \cite{Pu}, prior to the attempts to address this problem for $C^0$-generic homeomorphisms \cite{Akin3,CMN,H,PPSS}.
In the case of homeomorphisms, the notion of permanent periodic point replaces hyperbolicity as the key ingredient to deduce a $C^0$-generic lower semicontinuity of the closure of the space of periodic orbits. 
In the case of impulsive semiflows, the results on the abundance of periodic orbits under $C^0$-perturbations are of substantially different nature whenever one considers either the {flows} or the impulses. This is due to the fact that perturbations on {flows} have a global impact while perturbations of impulses affect only the impulsive region, and will not affect non-wandering points that do not meet the impulsive region (compare the statements of Theorems~\ref{thmA} and ~\ref{thmA-flows} and see the examples in Subsection~\ref{sec:examples}).

\medskip
 This paper is organized as follows. In Section~\ref{Impulsivesemi} we introduce the class of impulsive semiflows considered here, and the topologies 
 in the space of impulses and flows.
Section~\ref{sec:state} is devoted to the statement of the main results, considering perturbations of the impulses (Subsection~\ref{subsec:C0impulses}) 
and of the flows (Subsection~\ref{subsec:C0flows}). 
In Section~\ref{sec:prelimi} we introduce the concept of Poincar\'e 
map
 for this class of impulsive semiflows and prove several $C^0$-perturbation lemmas (including $C^0$-closing lemmas and creation of periodic attractors) both for flows and impulses. We also prove that  $C^0$-generic flows generate impulsive semiflows whose periodic orbits are all permanent, and that a similar conclusion holds for $C^0$-generic impulses, while considering periodic orbits that intersect the impulsive region.
The proofs of the main results are summoned in Sections~\ref{sec:proof1} and ~\ref{sec:proof2}. Finally, in Sections ~\ref{sec:examples} and ~\ref{sec:finalc} we provide several examples and discuss some possible future research directions, respectively.

%%%%%%%%%%%%%%%%%%%%%%%%%%%%%%%%%%%%%%

%%%%%%%%%%%%%%%%%%%%%%%%%%%%%%%%%%%%%%
\section{Setting}\label{Impulsivesemi}

In this paper, impulsive semiflows are determined by an initial continuous flow and an impulsive map. The concept of impulsive semiflow is introduced in Subsection~\ref{subsec:defimmpulsive}. 
In 
Subsections \ref{subsec:subsecspacemapspert}
and~\ref{subsec:subsecspaceflows} 
we describe the classes of spaces of impulsive maps and flows that will be considered.

\subsection{Impulsive semiflows}\label{subsec:defimmpulsive}

Let $M$ be a $d$-dimensional
Riemannian closed manifold ($d\ge 3$).
We will denote by $\text{dist}(\cdot,\cdot)$ the distance induced by the Riemannian structure.
A $C^0$ flow is a $C^0$ map $\varphi\colon M\times \mathbb{R}\rightarrow M$ such that 
\begin{itemize}
\item $\varphi(0,x)=x$ for all $x\in M$ and 
\item $\varphi(t+s, x)=\varphi(s,\varphi(t,x))$ for all $x\in M$ and $t,s\in\mathbb{R}$. 
\end{itemize}

Given a non-empty 
%\color{red} \color{red} compact \color{black} \color{black} 
set  $D\subset M$ we will refer to any continuous map $I:D \to M$ so that %{\color{magenta}{$\hat D\cap D=\emptyset$}}
{$\overline{I(D)} \cap \overline{D}=\emptyset$}
%{\color{forestgreen(web)}{$\text{dist}_H(D,I(D))>0$,
% where $\text{dist}_H$ stands for the Hausdorff distance,}}
 as an \emph{impulsive map}.  
 Under the above conditions we say that $(M,\varphi,D,I)$ is an \emph{impulsive dynamical system}  which generates an impulsive semiflow, possibly with discontinuities, as follows.
Let $\tau_1:M\to~(0,+\infty]$ be the first hitting time function to the set $D$ given by
\begin{equation}\label{eqtau1}
    \tau_1(x)=\left\{
\begin{array}{ll}
\inf\left\{t> 0 \colon \varphi_t(x)\in D\right\} ,& \text{if } \varphi_t(x)\in D\text{ for some }t>0;\\
+\infty, & \text{otherwise,}
\end{array}
\right.
\end{equation}
which is known to be upper semicontinuous \cite{C04}. 
The subsequent impulsive times are defined in terms of $\tau_1$ applied to the point obtained after the impulse. 
More precisely, setting
$$
\gamma_x(t)=\varphi_t(x), \qquad \forall \,  0\le t<\tau_{1}(x)  \quad \mbox{and} \quad \gamma_x(\tau_1(x))= I(\varphi_{\tau_{1}(x)}(x)),
$$
the remainder of the {impulsive trajectory} and {impulsive times} $(\tau_n(x))_{n\ge 2}$ 
(possibly finitely many) of the point $x\in M$ are defined recursively by 
\begin{displaymath}
\gamma_x(\tau_{n}(x))=I(\varphi_{\tau_n(x)-\tau_{n-1}(x)}(\gamma_x({\tau_{n-1}(x)})))
\end{displaymath}
and
\begin{displaymath}
\tau_{n+1}(x)=\tau_n(x)+\tau_1(\gamma_x(\tau_n(x))).
\end{displaymath}
In particular, 
\begin{equation}\label{def:eq:impulsive}
\gamma_x(t)=\varphi_{t-\tau_n(x)}(\gamma_x(\tau_n(x)))
\quad\text{whenever} \; \tau_n(x)\le t<\tau_{n+1}(x).
\end{equation}
%{\color{magenta}{Under the assumption that 
 %$\hat D\cap D= \emptyset$ 
% $\text{dist}_H(D,\hat D)>0$, it is easy to check that 
% $$
% \sup_{n\ge 1}\,\{\tau_n(x)\}=+\infty
%    \quad \text{for every $x\in M$}
% $$
% }}
 %$\hat D\cap D= \emptyset$ 
  We observe that under the standing assumption that 
  $\overline{I(D)} \cap \overline{D}=\emptyset$, it is easy to check that 
 $$
 \sup_{n\ge 1}\,\{\tau_n(x)\}=+\infty
    \quad \text{for every $x\in M$,}
 $$
 %(cf. \cite{AC14}, Remark 1.1), 
 hence the impulsive trajectories are defined for all positive $t$.

 \smallskip
 The \emph{impulsive semiflow} $\psi_{\varphi,I}$ associated to $(M, \varphi, D, I)$  is now defined as
\begin{equation}\label{def:eq:impulsive2}
\begin{array}{cccc}
        \psi_{\varphi,I}:  &  \mathbb{R}^+_0 \times M & \longrightarrow & M \\
        & (t,x) & \longmapsto & \gamma_x(t), \\
        \end{array}
\end{equation}
where $\gamma_x$ stands for the impulsive trajectory of $x$ determined by $(M,\varphi,D, I)$ and described by ~\eqref{def:eq:impulsive}. 
The arguments in   \cite[Proposition 2.1]{B} guarantee that
$\psi_{\varphi ,I}$ is indeed a piecewise continuous semiflow, whose trajectories are 
piecewise smooth. 

The main goal  of this paper is to describe the set of periodic orbits of generic  impulsive semiflows when one considers the dependence on the impulsive map $I$ and flow $\varphi$.
%\begin{remark}
%In order not to overload the notation and to highlight the dependence we are interested in, we %will denote the impulsive semiflow
%in \eqref{def:eq:impulsive2} simply by $%\psi_{{\varphi},I}$ or $\psi_{{\varphi},I}$ whenever the vector field or the impulse are fixed, respectively.
    %\end{remark}

 \smallskip
The \emph{orbit of} $x\in M$ by the impulsive semiflow $\psi_{\varphi,I}$ is the set $O_{x}=\{\gamma_x(t) \colon t\ge 0\}$. 
%Whenever necessary we shall write $O_{I,x}$ and $\gamma_{I,x}$ to emphasize the impulsive map $I$.
We say that
a point   $x \in M$ is \emph{non-wandering}  if, for every neighborhood $U$ of $x$ and any $T>0$, there exists $t \geq T$ such that 
\[
(\psi_{\varphi,I})_t^{-1}(U) \cap U \neq \emptyset.
\]
\color{black}
We denote by  $\Omega(\psi_{\varphi,I})$ the set of the non-wandering-points for the semiflow $\psi_{\varphi,I}$.

\subsection{Space of impulsive maps}
\label{subsec:subsecspacemapspert}

In this subsection we introduce the space of impulsive semiflows, parameterized by their impulsive maps.
Given  two smooth codimension one submanifolds $D, \hat D \subset M$ which are homeomorphic and so that $\overline{\hat D} \cap \overline{D}=\emptyset$, 
consider 
the space
\begin{equation}\label{eq:defimpT}
\mathscr I_{D,\hat D} = \text{Homeo}(D,\hat D)
\end{equation}
endowed with the $C^0$-distance.  We observe that $\mathscr I_{D,\hat D}$ is a Baire space.

%{\color{magenta} We will always consider $D, \hat D \subset M$  }

\medskip
In the context of perturbations of the impulsive maps on an impulsive region $D$ it is a natural condition to require that the
continuous flow 
 $\varphi$ is such that  
 $D$ is a codimension one 
smooth submanifold of $M$ transversal to the flow direction and that
\begin{equation}
    \label{eq:noncompact}\tag{H}
    \overline{D} \cap \text{Sing}(\varphi)=\emptyset,
\end{equation}
where 
$\text{Sing}(\varphi)$ stands for the equilibrium points of $\varphi$.

\smallskip
When $\varphi$ is a Lipschitz continuous flow satisfying ~\eqref{eq:noncompact}
and $I\in \mathscr I_{D,\hat D}$ then 
 there exists $\xi>0$ so that for every distinct, $x_1,x_2\in D\cup \hat D$ one has 
$
\{\varphi_t(x_1) \colon  |t|\le \xi\}\cap \{\varphi_t(x_2) \colon  |t| \le \xi\}=\emptyset.
$
Additionally,  given  $I\in \mathscr I_{D,\hat D}$, any discontinuity point of  $\psi_{\varphi,I}$ in the interior of $\hat D$  arises as a point mapped by the original flow $\varphi$ on the boundary of the impulsive region $D$.

\subsection{Space of  flows}\label{subsec:subsecspaceflows}

\medskip
One of our goals is to understand the dependence of an impulsive semiflow on the underlying flow, given the impulsive regions. Let us introduce some notation.
 Consider the space $\mathscr{F}^{0}(M)$ of continuous flows and the space $\mathscr{F}^{0,1}(M)$ of Lipschitz continuous flows generated by Lipschitz continuous vector fields in 
 $\mathfrak{X}^{0,1}(M)$.
We consider the compact-open topology $\tau$ in $\mathscr{F}^0(M)$ which is the smallest topology such that the map
\begin{equation*}
\begin{array}{crcl}
\Gamma\colon & \mathscr{F}^0(M) & \longrightarrow & C^0([a,b]\times M,M) \\& \varphi& \longmapsto & \varphi\Big|_{[a,b]\times M}
\end{array}
\end{equation*}
is continuous where, as usual, $C^0([a,b]\times M,M)$ is equipped with the distance:
\begin{equation}\label{d1}
\mathrm{d}_2\left(\varphi_1\Big|_{[a,b]\times M},\varphi_2\Big|_{[a,b]\times M}\right)=\max_{t\in[a,b],\,  x\in M}d(\varphi_1(x,t),\varphi_2(x,t)).
\end{equation}
It is well known that $\tau$ does not depend on $a,b$.

\medskip
Now, consider two smooth codimension one submanifolds  $D, \hat D \subset M$ so that 
$\overline{\hat D} \cap \overline{D} = \emptyset$ and let us define
the  space \begin{equation}
\label{def:spaceofVectorFields}
\mathfrak V_{D,\hat D} = \overline{\Big\{\varphi \in \mathscr{F}^{0,1}(M) \colon X= \dfrac{d}{dt}(\varphi(t,x))|_{t=0}, \  \overline{D} \pitchfork X \ \text{and} \  \overline{\hat D} \pitchfork X   \Big\}}, 
\end{equation}
endowed with the $\tau$-topology (here $\pitchfork$ is used as usual to define transversality). Note that the space $\mathfrak V_{D,\hat D} $ corresponds to the  closure 
in $\mathscr{F}^0(M)$
with respect to the $\tau$ topology of the space of  continuous flows, whose generating vector field is transversal to both $D$ and  $\hat{D}$. 
We point out that $\mathfrak V_{D,\hat D}$
is a Baire space (cf. Proposition~\ref{Baire}).

\color{black}

\section{Statement of the main results}
\label{sec:state}

This section is devoted to the statement of the main results, concerning the  
set of periodic orbits
of typical impulsive semiflows. In Subsection~\ref{subsec:C0impulses} we fix a Lipschitz continuous flow 
and consider the dependence on the impulsive map $I$.
In Subsection~\ref{subsec:C0flows} we assume that 
an impulsive map $I$ is given and  consider the dependence of the impulsive semiflow 
on the {flow}.
Throughout let $M$ be a compact Riemannian manifold of dimension $d\ge 3$.

%%%%%%%%%%%%%%%%%%%%%%%%%
\subsection{$C^0$-Baire generic impulses}
\label{subsec:C0impulses}

The first result concerns the set of periodic
orbits of impulsive semiflows obtained from a fixed Lipschitz continuous flow and typical impulsive maps.

\begin{maintheorem}\label{thmA}
Let $\varphi \in \mathcal F^{0,1}(M)$ 
%be a Lipschitz continuous flow generated by  $X \in \mathfrak{X}^{0,1}(M)$ 
and $D,\hat D\subset M$ be 
smooth codimension one submanifolds transversal to the flow direction 
such that ~\eqref{eq:noncompact} holds.
There exists a $C^0$-Baire generic subset 
$\mathscr R_\varphi
\subset \mathscr I_{D,\hat D}$
 of impulses
such that 
\begin{equation}
\label{eq:thmAeq}
%\overline{Per_h(\psi_{\varphi,I}) \cap \hat D} = %{\Omega(\psi_{\varph},I}) \cap \hat D 
\overline{Per(\psi_{\varphi,I})} \cap \mathring{D} 
= 
{\Omega(\psi_{\varphi,I}) \cap \mathring{D}} 
%= IR(\psi_{{\varphi},I}) \cap  \mathring{D}
    \end{equation} 
for every 
$I \in \mathscr R_\varphi$, 
where ${Per(\psi_{\varphi,I})}$ denotes the set of periodic orbits of $\psi_{\varphi,I}$,
and $\mathring{D}$ stands for the interior of $D$.
\end{maintheorem}

%\begin{remark}
Some comments are in order. 
Firstly, as $(\psi_{{\varphi},I})_t(x)=\varphi_t(x)$ for every $x\in D$ and for every small $t>0$, periodic orbits under $\psi_{{\varphi},I}$ never intersect $D$.  Secondly, the conclusion of Theorem~\ref{thmA} cannot be written using the landing region 
$\hat D$,
as there exist 
%However, equation ~\eqref{eq:thmAeq} in the conclusion of Theorem~\ref{thmA} does not imply a similar expression on the range $\hat D$ of the impulse. More precisely, there are examples of 
$C^1$-open sets of impulses for which the equality
$\overline{{{Per(\psi_{{\varphi},I})}} \cap \hat {D}} = \Omega(\psi_{{\varphi},I}) \cap \hat D$ fails (see e.g. \cite{STV1}).
The next result is a direct consequence of the continuous dependence on initial conditions for the Lipschitz continuous flow and the proof follows closely Corollary~B in \cite{STV1}.

\begin{maincorollary}\label{corA}
%\begin{maincorollary}\label{corA}
Given $\varphi \in\mathcal F^{0,1}(M)$ and 
$D,\hat D\subset M$ 
smooth codimension one submanifolds transversal to the flow direction
such that ~\eqref{eq:noncompact} holds,
the following hold:
\begin{enumerate}
\item 
if $I_0\in  \mathscr I_{D,\hat D}$ is such that
  %$\Omega(\varphi) \cap \partial( I_0(D))=\emptyset$
  $\Omega(\psi_{\varphi,I_0}) \cap \partial D =\emptyset$ 
 then there exist $\delta>0$,
an open neighborhood $\mathcal V$ of $I_0$ and  a Baire generic subset 
$\mathscr R\subset \mathcal V$ so that, for every   $I\in \mathscr R$
one can write the non-wandering set $\Omega(\psi_{{\varphi},I})$ as a (possibly non-disjoint) union
%$$
%\Omega(\psi_{{\varphi},I}) =  \overline{Per_h(\psi_{{\varphi},I})} \cup \, \Omega_2(I,%\varphi),
%$$
$ 
\Omega(\psi_{{\varphi},I}) =  \overline{Per(\psi_{{\varphi},I})} \cup \, \Omega_2(\varphi, D),
$ 
where %$\Omega_2(I, \varphi)\subset \Omega(\psi_{{\varphi},I})$ 
$\Omega_2(\varphi, D)\subset \Omega(\psi_{{\varphi},I})$ 
 is a $\varphi$-invariant set 
which does not intersect a $\delta$-neighborhood of the cross-section $D$. 
Moreover, the set $\Omega(\psi_{{\varphi},I})\setminus D$ is a $\psi_{{\varphi},I}$-invariant subset of $M$.
%
%
%the set $\Omega(\psi_{{\varphi},I})\setminus D$ \color{black} is a $\psi_{{\varphi},I}$-invariant subset of $M$.
\smallskip
\item if $\varphi$ is minimal then there exists a Baire generic subset $\mathscr R\subset \mathscr I_{D,\hat D}$  so that, for every $I\in \mathscr R$, the set of periodic orbits is dense in $\Omega(\psi_{{\varphi},I})$.
Moreover, the set $\Omega(\psi_{{\varphi},I})\setminus D$ is a $\psi_{{\varphi},I}$-invariant subset of $M$.
\end{enumerate}
\end{maincorollary}

\subsection{$C^0$-Baire Generic flows}
\label{subsec:C0flows}

Our next results concerns the 
set of periodic orbits
of impulsive semiflows obtained from a given impulsive map and typical continuous flow.

\begin{maintheorem}\label{thmA-flows}
Let  $D, \hat D \subset M$ be smooth  codimension one submanifolds and 
$I\in \mathscr I_{D, \hat D}$ be an impulsive map.
 There exists a $C^0$-Baire generic subset
$\mathscr{R}_I\subset \mathfrak V_{D,\hat D}$ such that
      $$
    \overline{\text{Per}(\psi_{\varphi,I})}=\Omega(\psi_{\varphi,I}), \quad \text{for every $\varphi \in \mathscr{R}_I$.}
 $$
\end{maintheorem}

Some comments are in order. First, it should be noticed that the impulsive semiflows are at most as regular as the impulsive map, hence rarely Lipschitz continuous. Second, the statement of Theorem~\ref{thmA-flows} is substantially stronger than the conclusion of Theorem~\ref{thmA}, and this is due to the fact that perturbing the vector field one can obtain denseness of periodic orbits even among points in the non-wandering set that do not visit the impulsive region. 

\bigskip

\color{black}
%%%%%%%%
\section{Technical results}\label{sec:prelimi}

This section is devoted to develop the main technical tools to be used in the proofs of the main results. We prove that the spaces of flows are Baire spaces, introduce Poincaré maps and prove several perturbation lemmas - both on flows and impulses - leading to the appearance of permanent periodic orbits.

\subsection{Baire spaces}\label{subsec:Poincare}

In what follows we establish that the spaces of continuous flows considered to parametrize the impulsive flows are Baire spaces.

\begin{proposition}\label{Baire}
The spaces $\mathscr{F}^0(M)$ %$\mathscr{F}(M)$
and $\mathfrak V_{D,\hat D}$  are  Baire spaces with respect to $\tau$.
\end{proposition}
\begin{proof}
We begin by 
noting
that the topology $\tau$ is metrizable through the distance
\begin{equation}\label{d2}
\mathrm{d}_1(\varphi_1,\varphi_2)=\max_{t\, \in [a,b], \,x\, \in M}d(\varphi_1(t, x),\varphi_2(t, x)).
\end{equation}
Note that $\Gamma(\mathscr{F}^0(M))\subset C^0( [a,b] \times M,M)$ and $C^0([a,b] \times M,M)$ is a complete metric space (cf. \cite[pp. 62]{Hi}). 
\begin{claim}\label{claim}
 $\Gamma(\mathscr{F}^0(M))$ is closed.
 \end{claim}

In this case $\Gamma(\mathscr{F}^0(M))$ is complete because a closed subset in a complete space is also complete. To see that $\mathscr{F}^0(M)$ is a complete space we take a Cauchy sequence $\{\varphi_n\}_n\in \mathscr{F}^0(M)$. It follows directly from the definition of the metrics $\mathrm{d}_1$ in \eqref{d2} and $\mathrm{d}_2$ in \eqref{d1} that $\Gamma(\{\varphi_n\}_n)$ is also Cauchy in $C^0( [a,b]\times M,M)$. Since $C^0( [a,b]\times M,M)$ is complete $\Gamma(\{\varphi_n\}_n)$ converges to $\hat\varphi\in C^0( [a,b]\times M,M)$. By Claim~\ref{claim} we obtain that $\hat\varphi\in \Gamma(\mathscr{F}^0(M))$ and so exists $\varphi\in \mathscr{F}^0(M)$ such that $\Gamma(\varphi)=\hat\varphi$. From $\lim_{n\to\infty} \mathrm{d}_2(\Gamma(\varphi_n),\hat\varphi)=0$ we obtain $\lim_{n\to\infty} \mathrm{d}_1(\varphi_n,\varphi)=0$.

We are left to prove Claim~\ref{claim}. 
Take any $\{\varphi_n\}\in \mathcal{F}^0(M)$ converging, in the metric $\mathrm{d}_1$, to $\varphi\in C^0( [a,b]\times M,M)$. We must prove that $\varphi\in \mathcal{F}^0(M)$. We have
\begin{equation}\label{dd6}
0=\lim_{n\to\infty}\mathrm{d}_1(\varphi,\varphi_n)=\lim_{n\to\infty}\max_{t\,\in\,[a,b], \,x\,\in \,M}d(\varphi(t,x),\varphi_n(t,x)).
\end{equation}
The equality $\varphi(0,x)=x$ for all $x\in M$ follows directly from \eqref{dd6} taking $t=0$ and noting that $\varphi_n(0,x)$ is constant since $\varphi_n$ is a flow. Now we prove the equality $\varphi(t+s, x)=\varphi(t,\varphi(s,x))$ for all $x\in M$ and all $t,s\in\mathbb{R}$. Fixing any $x$, it is sufficient to show that the group property holds for $t,s$ satisfying that:
\begin{equation}\label{ts}
 t\in[a,b)\text{ and} \ s \text{ is such that}\,\,\, t+s\in(a,b],  
 \end{equation}
 since then we can shift $x$ along its $\varphi_t$-orbit. Take $\varepsilon>0$ and choose $n_1$ such that, for all $n\geq n_1$ we have
\begin{equation}\label{dd6768}
\mathrm{d}_1(\varphi,\varphi_n)<\frac{\varepsilon}{3}.
\end{equation}
Since $\varphi$ and $\varphi_n$ are continuous, 
$\varphi(0,x)=x$ and $\varphi_n(0,x)=x$ for all $x\in M$ take $n_2>n_1$ such that 
\begin{equation}\label{dd6768tr}
d(\varphi_n(t, \varphi_n(s, x)),\varphi_n(t, \varphi(s, x)))<\frac{\varepsilon}{3}
\end{equation}
holds for all $n\geq n_2$ and $t,s$ as in \eqref{ts}.
Since $\varphi_n$ is a flow we get:
\begin{eqnarray*}
d(\varphi(t+s, x),\varphi(t, \varphi(s,x)))&\leq &d(\varphi(t+s, x),\varphi_n(t+s, x))\\
&+&d(\varphi_n(t, \varphi_n(s, x)),\varphi_n(t, \varphi(s, x)))\\
&+&d(\varphi_n(t,\varphi(s, x)),\varphi(t, \varphi(s,x)))\\
&\overset{\eqref{dd6768}+\eqref{dd6768tr}+\eqref{dd6768}}{\leq}&\frac{\varepsilon}{3}+\frac{\varepsilon}{3}+\frac{\varepsilon}{3}=\varepsilon.
\end{eqnarray*}
This proves that $\mathscr{F}^0(M)$ is a Baire space.
Since $\mathfrak V_{D,\hat D} $ is a closed subspace of
 $\mathscr{F}^0(M)$, it is also a Baire space.

\end{proof}

%%%%%%%%%

%%%%%%%%%
\subsection{Poincar\'e maps for impulsive semiflows}\label{subsec:Poincare}

Given $\varphi \in \mathscr{F}^{0,1}(M)$  generated by a vector field $X \in \mathfrak{X}^{0,1}(M)$, a periodic orbit $\gamma$ and a local cross-section $\Sigma$ passing through a point $p \in \gamma$, the Poincar\'e map
$P: U\subset \Sigma \to \Sigma$ is a bi-Lipschitz local homeomorphism defined on some open neighborhood $U\subset \Sigma$ of $p$ which is 
the first return map of points in $ U$ to $\Sigma$.  
%This map is  called  the Poincar\'e map and it a diffeomorphism. 

In the case of impulsive semiflows, by some abuse of notation we shall denote as Poincar\'e map the ones that concatenate the impulse and the first hitting time map to $D$. More precisely, given $I\in \mathscr I_{D,\hat D}$ we shall consider the map (which we denote as a Poincar\'e map)
\begin{equation}\label{defPo00}
\begin{array}{cccc}
P_{\varphi,I} : & \widehat{\hat D} & \to & \hat D \\
	& x & \mapsto & I \circ \varphi_{\tau(x)}(x)
\end{array}
\end{equation}
where 
\begin{equation}\label{deftaup}
\tau(x)=\inf\{t>0 \colon \varphi_t(x)\in D\} 
\end{equation} 
for every $x\in \hat D$ and 
$\widehat{\hat D}=\{ x\in \hat D \colon \tau(x) <\infty\}.$
Notice that $\tau=\tau_1\mid_{\hat D}$.
Consider the  set 
$$
\widehat{I_*(D)}=\{x\in \widehat{\hat D} \colon \varphi_{\tau(x)}(x) \notin \partial D \}
$$
which consists of the points in $\widehat{\hat D}$ which are mapped in the interior of $D$.
It will be also convenient to consider the map
$f_{\varphi,I}: D \to D$ given by
\begin{equation}
    \label{eq:deffI}
    f_{\varphi,I}= I^{-1} \circ P_{\varphi,I} \circ I,
\end{equation}
defined on the impulsive region $D$, which is the support of all perturbations of the 
impulsive maps.
%By some abuse of notation, for notational simplicity we shall write $P_I$ and $f_I$ (resp. $P_X$ and $f_X$) whenever the vector field $X$ (resp. impulse $I$) is fixed and we want to study the dependence on the impulse $I$ (resp. vector field $X$).

\medskip
%We will make use of the followign flowbox theorem for Lipschitz vector fields:
%\begin{lemma}\label{thm:BC} \cite[Theorem~4]{BC}
%If $p$ is a regular point for a Lipschitz vector field $X\in \mathfrak{X}^{0,1}(M)$,
% $(X^t)_t$ is the flow generated by $X$ and  $(T^t)_t$ is the trivial flow on $\mathbb R\times \mathbb R^{n-1}$ defined by $T^t(s,z)=(s+t,z)$ then the following holds: there exists $\delta>0$, a  neighborhood $U_\delta$ of $p$ and a lipeomorphism 
%$\varphi: \overline U_\delta \to [0,\delta] \times \overline{B(\vec0,\delta)}$ (here $B(\vec0,\delta)$ denotes the usual ball in $\mathbb{R}^{n-1}$) so that $\varphi \circ X^t (x) = T^t \circ \varphi(x)$ for every $x\in U_\delta$ and $t\in \mathbb R$ so that $X^t(x)\in \overline{U_\delta}$.
%\end{lemma}

%Let $\mathscr I_{D,\hat D}$ be the space of homeomorphisms from $D$ to $\hat D$. 
 Let us recall the Lipschitz flowbox theorem.

\begin{lemma}\label{thm:BC} \cite[Theorem~4]{BC}
Let  $X\in \mathfrak{X}^{0,1}(M)$ be a Lipschitz continuous vector field,
 $\varphi$  the flow generated by $X$ and  $T$ the trivial flow on $\mathbb R\times \mathbb R^{d-1}$ defined by $T_t(s,z)=(s+t,z)$. 
 If $p$ is a regular point then there exist $\delta>0$, a neighborhood $U_\delta$ of $p$ and a lipeomorphism 
$h: U_\delta \to [0,\delta] \times {B(\vec0,\delta)}$ 
%(here $B(\vec0,\delta)$ denotes the usual ball in $\mathbb{R}^{n-1}$) 
so that $h \circ \varphi_t (x) = T_t \circ h(x)$ for every $x\in U_\delta$ and $t\in \mathbb R$ so that $\varphi_t(x)\in \overline{U_\delta}$.
\end{lemma}

In what follows we prove that, even though the Poincar\'e map in \eqref{defPo00} is defined by expressions which take into account the impulse  it is a local homeomorphism.

\begin{proposition}\label{prop:smoothnessP}
Let $\varphi\in\mathscr{F}^{0,1}(M)$ and smooth submanifolds $D,\hat D$ of codimension one transversal to the flow direction  so that 
$\overline{\hat {D}} \cap \overline{D}=\emptyset$. 
If \eqref{eq:noncompact} holds, 
$I \in \mathscr I_{D,\hat D}$ and $x \in \widehat{I_*(D)}$
%$p\in \gamma \cap D$,  $\hat p = I(p) \in (\gamma \cap \hat{D}) \setminus \partial \hat{D}$. 
then there exist a local cross-section $\hat{\Sigma}\subset \hat D$ and  a neighborhood  $\hat{V}\subset \hat{\Sigma}$ of $x$  such that the Poincar\'e map 
$P_{\varphi,I}|_{\hat{V}} $ 
is an homeomorphism onto its image. 
\end{proposition}

%{\color{forestgreen(web)}
%{Let $\varphi$ be a Lipschitz continuous flow generated by  $X\in \mathfrak{X}^{0,1}(M)$, $D,\hat D$ be smooth submanifolds of codimension one transversal to the flow so that \eqref{eq:noncompact} holds, and let $I \in \mathscr I_{D,\hat D}$ be an  impulse. 
%Given $x \in \widehat{I_*(D)}$
%%$p\in \gamma \cap D$,  $\hat p = I(p) \in (\gamma \cap \hat{D}) \setminus \partial \hat{D}$. 
%there exist a local cross-section $\hat{\Sigma}\subset \hat D$ and  a neighborhood  $\hat{V}\subset \hat{\Sigma}$ of $x$  such that the Poincar\'e map $P{_I}|_{\hat{V}} $ is an homeomorphism onto its image. 
%}}

\begin{proof}
%Suppose that $\bar{\tau}$ is the period of  the periodic orbit $\gamma$. Given $p\in \gamma \cap D$, 
%Since $D$ and $ \hat{D}$ are both \color{red} compact \color{black} submanifolds such that  $D\cap \hat{D}= \emptyset$  
By the assumption $\overline{\hat {D}} \cap \overline{D}=\emptyset$,
there is an open neighborhood $W$ of $\hat{D}$ such that  $\psi_{{\varphi},I} |_{W} = \varphi$. 
Let $x \in \widehat{I_*(D)}$ and consider a local cross-section  $\hat{\Sigma}\subset \hat D$. %By  transversality, 
The Lipschitz flowbox theorem for $\varphi$ 
(cf. Lemma~\ref{thm:BC})
guarantees the existence of  a flow-box $(V_1, h_1)$ at $x$  where $V_1 \subset W $ is an open neighborhood of $x$,  
and the chart map $h_1: V_1\to [0,1]\times U_1 \subset[0,1]^d$ sends trajectories of $X$ in $V_1$  into trajectories of the vector field 
$X_{V_1}(z):= (1,0, \cdots, 0)$ for every $z\in [0,1]^d$.

Choose $t_0>0$ small such that $\varphi_{t_0}(x)$ still lies in $V_1$  
and consider the compact segment of orbit  $\Gamma:= \{ \varphi_t(x):  t_0 \le  t \le  \tau(x)\}$. Since   $\tau(x)< \infty$, 
there exists a long tubular {Lipschitz} flow-box chart $(V_2, h_2)$ for the original flow $\varphi$, say
$h_2: V_2\to [0,1]\times U_2 \subset[0,1]^d$
 such that $V_2 \supset \Gamma$. 
Let $\Sigma_1=h_2^{-1}(\{0\} \times U_2)$ and $\Sigma_2=h_2^{-1}(\{1\} \times U_2)$ be the components of the boundary of $V_2$ which are transversal to the flow $\varphi$.  Denote by $\pi_1: \hat{V} \subset \hat{\Sigma}\to \Sigma_1$,  $\pi_2: \Sigma_1 \to  \Sigma_2$,  the projections along the positive trajectories of $\varphi$ given by the tubular flowbox theorems, where $\hat{V}$ is a small open neighborhood of $x$ in $\hat{\Sigma}\cap \widehat{I_*(D)}$. Similarly, let  $\pi_3: \Sigma_2 \to  \Sigma$ be the projection along the negative trajectories of $\varphi$, where $\Sigma$ is a small neighborhood of $\varphi_{\tau(x)}(x)$ in $D$. 
Each $\pi_i$, $i=1, 2, 3$ is a Lipschitz continuous map. 
Since
$$
P_{\varphi,I}|_{\hat{V}}\equiv I\circ \pi_3\circ \pi_2\circ\pi_1
$$ 
and the 
impulsive map $I$ is an homeomorphism we have that 
$P_{\varphi,I}|_{\hat{V}}$ 
is continuous. 
Using the invertibility of the impulsive map $I$ and the flow $\varphi$, the same argument can be done for construct the inverse of the Poincar\'e map as before, proving that it is a local homeomorphism. 
\end{proof}

%%%%%%%%%%%%%%%%%%%%%%%%%%

\subsection{Perturbations lemmas}\label{closing&homo}

\subsubsection*{$C^0$-perturbations of impulses }
We consider first the case where the flow in $\mathscr{F}^{0,1}(M)$ is fixed and the $C^0$-perturbations are considered in the space of 
impulsive maps. 

\medskip
We will say that a periodic orbit $\gamma$ for 
an
 impulsive semiflow 
$\psi_{{\varphi},I}$ is an \emph{attracting periodic orbit} if there exists a local cross-section $\Sigma$ and $p\in \Sigma\cap \gamma$ that is an attracting fixed point for the Poincar\'e map 
$P_{\varphi,I}$.
\bigskip

\begin{lemma}\label{cl2} ($C^0$-closing lemma: impulses) Let  $D$, $\hat D$ be smooth submanifolds of codimension one and let $\varphi \in \mathscr F^{0,1}(M)$. Assume that $I\in \mathscr I_{D, \hat D}$ is an impulsive map and 
 $p\in \Omega(\psi_{\varphi,I})$  is a regular point so that 
 $\overline{\gamma_p} \cap D\neq \emptyset$ and  $\gamma_p\cap \partial \hat D=\emptyset$. 
For every $\varepsilon>0 $ there exists 
$J\in \mathscr I_{D, \hat D}$ which is $\varepsilon$-$C^0$-close to $I$
%, coincides with $J$ outside of a small open neighborhood of $p$, 
 so that $\psi_{\varphi,J}$ has an attracting periodic orbit 
$x$ such that $\text{dist}(p,x)<\varepsilon$.

\end{lemma}

\begin{proof}
We may assume that $p\in \Omega(\psi_{{\varphi},I})$ is not periodic for $\psi_{{\varphi},I}$.
If $p\in \hat D$ is a periodic point for $\psi_{\varphi,I}$ and 
$P_{\varphi,I}^n(p)=p$ 
then there exists an arbitrary $C^0$-small perturbation $J$ of $I$ in such a way that $p$ is an attracting periodic orbit for $\psi_{\varphi,J}$ (see e.g. \cite[Proposition 4.3]{AABT}).
Otherwise, up to consider a point on the orbit of $p$, we may assume without loss of generality that  $p\in \Omega(\psi_{{\varphi},I}) \cap \hat D$  is non-periodic. 
Observe that $p$ is a non-wandering point for the Poincar\'e map
 $P_{\varphi,I}$ 
given by ~\eqref{defPo00}.
In particular, $q=I^{-1}(p)$ is a non-wandering point for $f_I: D \to D$ expressed by ~\eqref{eq:deffI}.

Given $\vep>0$, there exist $x\in \hat D$ and a least integer $n\ge 1$ such that $d(x,p)<\vep$ and
$d(x_n,p)<\vep$, where 
$x_i=P_{\varphi,I}^i(x)$.
Noting that $\hat D$ is of dimension $d-1\ge 2$, that $\{ {x_i} : 1\le i \le n\}$ is a finite collection
of pairwise disjoint points and $d(x_n,p)<\vep$, there exists an homeomorphism $\zeta: \hat D \to \hat D$ homotopic to the identity %can be time 1 of a flow
such that $\zeta(x_n)=x$, $\zeta(x_i)=x_i$ for every $1\le i \le n-1$ and $d_{C^0}(\zeta,\text{id})<\vep$.
The impulsive map $J=\zeta \circ I$ is $\vep$-$C^0$-close to $I$ and $x\in \text{Per}(\psi_{\varphi},J)$. 
This finishes the proof of the lemma.
\end{proof}

\smallskip

\subsubsection*{$C^0$-perturbations of flows}
In the notation of the flowbox theorem (recall Lemma~\ref{thm:BC}), given $0<r<\delta$ we will consider also the flowbox `cylinders' 
$\mathcal{F}_{X,\delta}(B):={h^{-1}} \ ( [0,\delta] \times B)$, for $B\subset B(\vec0,\delta)$, and the local 
cross-section $\Sigma_p:={h^{-1}} \ (\{0\} \times B(\vec0,\delta))$ to the flow $\varphi$ at $p$. 
Since the points in the impulsive region $D$ do not belong to the orbits of the impulsive semiflow, the latter provides a local trivializing chart for regular points of an impulsive semiflow whose orbit closure does not intersect $\partial D$.  
The following perturbation lemma will be instrumental.

\begin{lemma} \cite[Lemma~5.3]{BTV}\label{cl}
Let $X\in \mathfrak{X}^{0,1}(M)$, let $p$ be a regular point, $\Sigma_p$ be a local cross-section through $p$ 
and let $h$ be the lipeomorphism given by Lemma~\ref{thm:BC}. 
Assume there exists a $C^1$-arc $\gamma\colon [0,\delta]\rightarrow \Sigma_p$ with constant speed $\nu$ 
(hence of length $\nu \delta$) such that $\gamma(0)=p$ and $\gamma(\delta)=q$. 
Then there exists $C>0$ so that for any $\eta>\nu\delta$, any open tubular neighborhood  $B\subset \Sigma_p$ 
of \text{diam}eter smaller than $\eta$ containing $\gamma$ 
and any $\vep>0 $ there exists a vector field $Y$  such that:
(a) $Y\in \mathfrak{X}^{0,1}(M)$;
(b) $Y$ is $C \eta(1+\vep)$-$C^0$-close to $X$;
(c)  
$\varphi^Y_\delta(p)=\varphi^X_\delta(q)$
%$Y^{\delta}(p)=X^\delta(q)$ 
and
(d) $Y=X$ outside $\mathcal{F}_{X,\delta}(B)$.
\end{lemma}

\color{black}
The previous perturbation lemma is the key ingredient in the proof of the following $C^0$-closing lemma for impulsive semiflows, when perturbing the Lipschitz vector fields.

\begin{lemma}\label{cl2vectors} ($C^0$-closing lemma: vector fields)
Let $\varphi\in \mathcal F^{0,1}(M)$ be generated by $X\in \mathfrak{X}^{0,1}(M)$, $D$, $\hat D$ be smooth submanifolds of codimension one transversal to the flow direction and $I\in \mathscr I_{D, \hat D}$ be an impulsive map. If $p\in \Omega(\psi_{{\varphi},I})$ is a regular point so that $\gamma_p\cap \partial \hat D=\emptyset$, then for every $\varepsilon>0 $ there exists  $Y\in \mathfrak{X}^{0,1}(M)$ which is $\varepsilon$-$C^0$-close to $X$, coincides with $X$ outside
of a small open neighborhood of $p$, and such that 
the impulsive semiflow 
$\psi_{\phi,I}$, for the Lipschitz flow $\phi$ generated by $Y$,
has a periodic orbit 
$\tilde p$ so that $d(p,\tilde p)<\varepsilon$.
\end{lemma}

\begin{proof}
In case $p$ is periodic for $\psi_{{\varphi},I}$ we are done.
Hence, we may assume that $p$ is non-periodic.
%and 
%$\gamma_p\cap \hat D \neq \emptyset$ (otherwise the conclusion of the lemma follows from \cite[Corollary~5.4]{BTV}). 
%
We subdivide the proof in two cases. First, if $p\in \Omega(\psi_{{\varphi},I}) \setminus D$ is a regular point, for any 
$0<r<\delta$ and corresponding flowbow $\mathcal{F}_{X,\delta}(B(\vec0,r)) = h_1^{-1}( [0,\delta] \times B(\vec0,r))$  associated to $p$ there exists $T>0$ such that 
$
\gamma_{T}(\Sigma_p)\cap \Sigma_p \not=\emptyset
$ 
where $\Sigma_p=h_1^{-1}(\{0\} \times B(\vec0,r))$.
Reduce $\delta>0$, if necessary, so that 
$\mathcal{F}_{X,\delta}(B(\vec0,r))\cap D=\emptyset$.
Consequently, just pick $0<r<\delta$ small (it is enough that $4Cr < \vep$) so that the perturbed vector field 
$Y\in \mathfrak{X}^{0,1}(M)$ guaranteed by Lemma~\ref{cl} is $\vep$-$C^0$-close to the vector field $X$ and 
%{\color{darkorange}{$\psi_Y$
%admits a periodic orbit that intersects $\mathcal{F}_{X,\delta}(B(\vec0,r))$.}}
the impulsive semiflow $\psi_{\phi,I}$, for the flow $\phi$ generated by $Y$,
admits a periodic orbit that intersects $\mathcal{F}_{X,\delta}(B(\vec0,r))$.
Second, if $p\in \Omega(\psi_{{\varphi},I}) \cap D$ then there exists a small $t>0$ so that 
$q=X_{-t}(p) \in \Omega(\psi_{{\varphi},I}) \setminus D$, which reduces this case to the previous context.
This finishes the proof of the lemma.
\end{proof}

The next technical lemma will allow to make perturbations of continuous flows and 
give
rise to Lipschitz continuous flows (generated by Lipschitz continuous vector fields) whose orbits belong to the interior of certain topological balls. 
  
 \begin{lemma}\label{mario}
Let $B\subset M$ be an open ball, $p\in \partial B$ and $\varphi \in \mathscr{F}^0(M)$ such that $\varphi_\Pi(p)=p$ for some $\Pi>0$. For any $\vep>0$, there exists $X \in\mathfrak{X}^{0,1}(M)$ such that 
the Lipschitz flow $\phi$ generated by $X$
satisfies:
\begin{itemize}
\item [(i)] ${\rm d_2}(\varphi\big|_{[a,b]\times M},\phi\big|_{[a,b]\times M})<\vep$,
\item [(ii)] $\phi_{\ell}(q)=q$ for some $q\in B$ and $|\ell-\Pi|<\vep$.
\end{itemize}
In particular, $\phi$ is $\vep$-$\tau$ close to $\varphi$. 
\end{lemma}
\begin{proof} 
We begin by considering $\hat X \in\mathfrak{X}^{0,1}(M)$ so that 
$d_2(\varphi\big|_{[a,b]\times M},\hat \phi \big|_{[a,b]\times M})$ is arbitrarily close to 0, where 
$\hat \phi$ is the flow generated by 
the vector field $\hat X$, and $\hat{X}(p)$ is transversal to $T_pB$.
Using local charts, we may assume without loss of generality that a neighborhood of $p$ can be identified with 
 an open subset of $\mathbb{R}^d$. 
%We consider the case 
%that the orbit of $p$ does not intersect $B$.
%In particular, given $\eta>0$ it follows that  $\varphi_{t}(p)\notin B$ for all $|t|<\eta$.
%%
%Using local charts, we may assume without loss of generality that the ambient space is an open subset of $\mathbb{R}^d$. 
%endowed with the inner structure $\langle \cdot,\cdot\rangle$ whenever we perform our perturbation, say `near' the point $p$.
Let $\Sigma_p$ be a codimension-one local cross-section at $p$ transversal to $\hat X(p)$ of diameter smaller than $\delta$, let $B^{d-1}(p,\delta)\subset \Sigma_p$ be the codimension-one ball of radius $\delta$ around $p$ in $\Sigma_p$ and denote by $\partial B^{d-1}(p,\delta)$ the boundary of the latter. Note that, reducing $\delta$, we obtain the flowbox
$$\mathcal{F}_{\hat \phi,\delta}\left(B^{d-1}(p,\delta),\left[-\frac{3\Pi}{8},\frac{3\Pi}{8}\right]\right)=\left\{\hat \phi_t (B^{d-1}(p,\delta))\colon -\frac{3\Pi}{8}\leq t\leq\frac{3\Pi}{8} \right\}$$ 
which contains the orbit segment $\{\varphi_{t}(p) \colon 
-\frac{\Pi}{4}\leq t\leq\frac{\Pi}{4} \}$  in its interior. 
By continuity of the flow $\hat \phi$, one can choose a point 
$$
q\in [B\cap B^{d-1}(p,\delta)] \setminus \partial B^{d-1}(p,\delta)
$$
arbitrarily close to $p$, in such a way that: (i) 
$\hat \phi_\ell(q)\in [B\cap B^{d-1}(p,\delta)]\setminus \partial B^{d-1}(p,\delta)$ where $\ell>0$, close to $\Pi$, is the first hitting time of the $\hat \phi$-orbit of $q$ in $\Sigma_p$; and (ii) the 
orbit segment $\{\hat \phi_{t+\ell}(q) \colon -\frac{\Pi}8\le t\le 0\}$ 
is contained in the interior of $\mathcal{F}_{\hat \phi,\delta}(B^{d-1}(p,\delta),[-\frac{\Pi}{8},0])$.

Now, in order to make perturbations in the vector field we use the Lipschitz flowbox theorem \cite{BC} and trivialize the flowbox 
$$\mathcal{F}_\delta:=\mathcal{F}_{\hat \phi,\delta}\left(B^{d-1}(p,\delta),\left[-\frac{\Pi}{2},0\right]\right)$$ as it ensures the existence of a lipeomorphism $h\colon \mathcal{F}_\delta\rightarrow [-\frac{\Pi}{2},0]\times B^{d-1}(\vec0,\delta)$ so that $h\circ \hat \phi_t=T_t\circ h$, whenever the latter is defined, where
$$
\begin{array}{cccc}
T_t \colon \, & [-\frac{\Pi}{2},0]\times B^{d-1}(\vec0,\delta) & \longrightarrow &[-\frac{\Pi}{2},0]\times B^{d-1}(\vec0,\delta) \\
		& (s, x_2,\dots,x_d) & \longmapsto  & (s+t, x_2,\dots,x_d)
\end{array}
$$
 is the flow generated by the vector field $T=(1,0,\dots,0)$. Consider the points $q_1=h(q)$ and $q_2=h(\hat \phi_\ell(q))$ in the cross-section $\{0\} \times B^{d-1}(\vec0,\delta)$, its middle point $z=\frac{q_1+q_2}{2}$ and $r=d(q_1,q_2)<\delta$.  By construction, $B^{d-1}(z,r)\subset B^{d-1}(\vec0,\delta)$. Fix $\rho>1$  such that $\rho r<\delta$.
The support of our perturbation will be the flowbox
\begin{equation}\label{flow}
\mathcal{F}:=
%\mathcal{F}_{T_t,\rho r}\left(B^{d-1}(z,\rho r),\left[-\frac{\Pi}{8},0\right]\right)=
\left\{T_t(B^{d-1}(z,\rho r))\colon -\frac{\Pi}{8}\leq t\leq0 \right\}
\subset h(\mathcal{F}_\delta)
\end{equation}
for the flow $T_t$.
 %In these flowbox coordinates we have $(t,\overbrace{x_2,\dots,x_{d}}^{\in B^{d-1}(\vec0,\delta)})$ 
%for $-\frac{\Pi}{8}\leq t\leq 0$. 
We can suppose without lost of generality that $z=\vec{0}$.
Recall that $q_1$ and $q_2$ are diametrically symmetric.
Hence there exists a two-dimensional disk $Z\subset \{0\} \times B^{d-1}(\vec0,\delta)$ containing $q_1,q_2$  in such a way that $q_1=R_{\pi}(q_2)$, where $R_\alpha$ stands for the rotation of angle $\alpha$ inside the disk $Z$.

We proceed to construct a vector field of the form $P=T+\tilde R$ whose flow $P_t$ is $d_2$-close to $T_t$ and such that 
\begin{equation}\label{final}
P_{\frac{\Pi}{8}}\left(-\frac{\Pi}{8}, \tilde q_2 \right)=(0, \tilde q_1)=q_1,
\end{equation}
where we write $q_i=(0, \tilde q_i) \in\{0\} \times B^{d-1}(\vec0,\delta)$, for $i=1,2$.
Denote by $\langle Z \rangle$ the plane that contains the disk $Z$. Consider the vector field $R$ on $\mathcal F$ defined by
$$
R(s,x_2, \dots, x_d) = \frac{16 \pi}{r \Pi} \| (x_2, \dots, x_d)\| \cdot (0, v_2, v_3, \dots, v_d)\in \{(s,0,0, \dots, 0)\} + \langle Z \rangle
$$
for every $(s,x_2, \dots, x_d)\in \mathcal F$,
where $(0,v_2, v_3, \dots, v_d) \in \langle Z \rangle$ is chosen to be the (unique) unit vector orthogonal to $(0,x_2, \ldots, x_d)$ that induces a positive orientation on $\langle Z \rangle$.
The flow $R_t$ generated by the vector field $R$ consists of rotations (with angles proportional to the distance to the orbit of $z$) in planes parallel to $\langle Z \rangle$. Note that $R_t(s,0, 0, \dots,0)=(s,0, 0, \dots,0)$ for every $t\in \mathbb R$ and, using that $\|\tilde q_1\|=\|\tilde q_2\|=\frac{r}2$, one obtains that $R_{t}(s,\tilde q_1)$ is the point obtained by rotation of angle $\frac{8t\pi}{\Pi}$ centered at $(s,0,0, \dots, 0)$ along the plane $\{(s,0,0, \dots, 0)\} + \langle Z \rangle$. In particular, 
$R_{\frac{\Pi}{8}}(s,\tilde q_1)=(s,\tilde q_2)$ for every $s$.
Take a $C^\infty$ bump-function $\alpha:[0,+\infty) \to [0,1]$ so that $\alpha(x)=1$ when $x\leq r$ and $\alpha(x)=0$ when $x\geq \rho r$, and consider the smooth vector field
$$
\tilde R
(s,x_2,\dots,x_{d})=\left\{
\begin{array}{cl}
      \alpha(\|(x_2,\dots,x_{d})\|)\, R(s,x_2,\dots,x_{d}) & \text{inside the flowbox} \,\,\,\mathcal{F}\\
      0 & \text{in}\,\,\,\mathcal{F}_\delta\setminus \mathcal{F}
\end{array} 
\right..
$$
As $P=T+\tilde R$ coincides with $T$ outside the flowbow $\mathcal{F}$ and the flowbox $\mathcal{F}$ is thin (depending on $\delta$) we get $P_t$ is $d_2$-close to $T_t$.
Moreover, as $[T,\tilde R]=0$ (the Lie brackets being zero is equivalent to say that the generating flows commute) then integrating $P=T+\tilde R$ on $\mathcal F$ one concludes that 
\begin{align*}
P_t(s,x_2,x_3, \dots,x_{d})
    & = R_t T_t (s,x_2,x_3, \dots,x_{d})  
    \\
    & = R_t (t+s,x_2,x_3, \dots,x_{d})  
\end{align*}
for any $t\geq0$ such that $s+t\leq0$.
Thus, 
$$
P_{\frac{\Pi}{8}}\Big(-\frac\Pi8,\tilde q_2\Big)=
R_{\frac{\Pi}{8}} (0,\tilde q_2)
= (0,\tilde q_1) =q_1.
$$
Using the Lipeomorphism $h$ we obtain a flow $\phi$ on $M$ which is 
$d_2$-close to $\varphi$ satisfying $h\circ \phi_t=P_t\circ h$ for all points in $h^{-1}(\mathcal F)\subset M$.
Moreover, using ~\eqref{final} and the fact that $h$  is an homeomorphism one concludes that 
\begin{align*}
h\circ \phi^{\ell}(q)& 
=P_{\ell}\circ h(q)
=P_{\ell}(q_1) 
= P_{\frac\Pi8} (  P_{\ell-\frac\Pi8}(q_1)) \\
& 
= P_{\frac\Pi8} ( {\hat \phi}_{\ell-\frac\Pi8}(q_1)) =
P_{\frac\Pi8} ( {\hat \phi}_{-\frac\Pi8}(q_2))= P_{\frac\Pi8} (-\frac\Pi8, \tilde q_2) = q_1 = h(q) 
\end{align*}
and, consequently, $\phi_{\ell}(q)=q$.
This finishes the proof of the lemma.
\end{proof}
%
%
%We are left with the case where the $\varphi$-orbit of $p$ intersects the open ball $B$, i.e., there exists  $q\in B$ in the orbit of $p$.
%Let $\hat{X} \in\mathfrak{X}^{0,1}(M)$ with associated flow $\hat \phi$ 
%arbitrarily close to $\varphi$ such that ${\hat \phi}_{\Pi}(q)$ is arbitrarily close to $q$.  Proceeding as in the previous case, we can find 
%$X \in\mathfrak{X}^{0,1}(M)$ with  associated flow $\phi$ satisfying $(i)$ and displaying $q$ as a periodic orbit satisfying $(ii)$.
%
%This finishes the proof of the lemma.

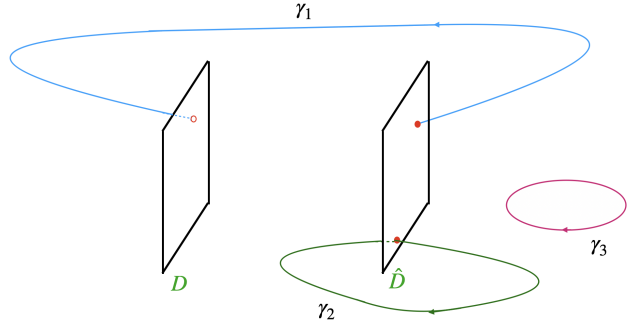
The previous lemma offers a perturbation leading to the existence of periodic orbits for the impulsive semiflow.
In what follows we guarantee that, up to small $\tau$-perturbations, such periodic orbit can be assumed to be attracting.

\begin{lemma}
\label{pert1}
Let $\varphi\in \mathcal F^{0,1}(M)$ be generated by $X\in \mathfrak{X}^{0,1}(M)$, $D$, $\hat D$ be smooth submanifolds of codimension one transversal to the flow direction and $I\in \mathscr I_{D, \hat D}$ be an impulsive map.
Assume that  $p\in  \text{Per}(\psi_{{\varphi},I})$  and $\gamma_p \cap \partial \hat D = \emptyset$. There exists a $\tau$-small perturbation $\phi \in \mathcal  F^{0,1}(M)$ of $\varphi$ such that $p$ is an attracting periodic point for $\psi_{\phi,I}$. 
\end{lemma}

\begin{proof}
Fix $\varphi\in \mathcal F^{0,1}(M)$ generated by $X\in \mathfrak X^{0,1}(M)$
 and $I\in \mathscr I_{D,\hat D}$. 
Assume that $\gamma_p\cap \partial \hat D=\emptyset$ and $\gamma_p\cap \hat D \neq \emptyset$
(the case $\bar{\gamma}_p \cap D = \emptyset$ is simpler and left to the reader).
Write
$\gamma_p\cap \hat D=\{y_0, y_1, \dots, y_{k-1}\}$ and $x_j=I^{-1}(y_j)$ for $0\le j \le k-1$. 

\smallskip
%Given $\vep>0$ let $\delta=\delta(\vep)>0$ be given by uniform continuity of $I^{-1}$.
%Fix $\vep>0$.
As the orbit $\gamma_p$ does not intersect $\partial \hat D$, $I^{-1}$ is 
uniformly continuous
and the space of $C^1$ vector fields is $C^0$-dense in the space of Lipschitz vector fields, given $\vep>0$, 
there exists $Y^1\in \mathfrak{X}^1(M)$ (generating the flow $\varphi^1$)
that is $C^0$-sufficiently close to $X$ in such a way that the first $k$ intersections of the forward orbit $\gamma^1_p$ of $p$, by the impulsive semiflow $\psi_{Y^1}$, with the cross-section $\hat D$ consists of the points 
$\{\bar y_0, \bar y_1, \dots, \bar y_{k-1}\}$
and their pre-images $\bar x_i=I^{-1}(\bar y_i) \in D$ satisfy 
$f_{\varphi^1,I}^i(\bar x_i)=\bar x_{i+1}$ for every $0\le i <k-1$ and 
$\text{dist}_{D}(f_{\varphi^1,I}(\bar x_{k-1}),\bar x_0) <\vep$.
Hence, one can perform a $\vep$-$C^0$-small perturbation $Y^2\in \mathfrak{X}^1(M)$ of $Y^1$ (generating the flow $\varphi^2$)
supported outside of a neighborhood of $p$
in such a way that $f_{\varphi^2,I}^k(\bar x_0)=\bar x_0$, hence $p\in  \text{Per}(\psi_{\varphi^2,I})$.

\smallskip
We proceed with the proof by choosing an appropriate flowbox for a compact piece of orbit of the flow $\varphi^2$ containing the point $\bar x_{0}$. More precisely, by the flowbox theorem
there exist $\delta>0$, $U_{0}\subset M$ and a trivializing diffeomorphism $h_{0} : U_{0} \to [0,\delta] \times {B(\vec0,\delta)}$ so that $h_{0}(\varphi^{2}_{-\delta}(\bar x_0))=(0,\vec0)$ and $h_{0}(\bar x_0)=(\delta,\vec0)$.
%, {\color{magenta}{as in Figure~\ref{fig:temp} below}}. 
Diminish $\delta>0$ if necessary in such a way that $x_i \notin U_0$ for every $1\le i \le k-1$.

\iffalse
\begin{figure}[htb]
    \centering
\includegraphics[width=0.8\linewidth]{flowboz.png}
    \caption{Flowbox on the last lap of the periodic orbit}
    \label{fig:temp}
\end{figure}
\fi

%by the flowbox theorem (Lemma~\ref{thm:BC}), for each $0\le j \le k-1$ there exist $\delta>0$, and neighborhoods $U_j$ of the orbit segment from $y_j$ to $x_{j+1 (\text{mod k})}$ and
%lipeomorphisms $h_j : U_j \to [0,\delta] \times {B(\vec0,\delta)}$ so that $h_j(y_j)=(0,\vec0)$ and $h_j(x_{j+1 (\text{mod k})})=(\delta,\vec0)$ for every $0\le j \le k-1$. 
{\color{red}{}}
The strategy will be to make a $C^0$-small perturbation of $Y^2$ to obtain a $C^1$-vector field $Y^3$, generating the flow $\varphi^3$, whose Poincar\'e map is a contraction. In order to do so, let us first construct the perturbation in local flowbox coordinates.

Given $0<\eta \ll \delta<1$ consider the $C^1$ bump-function $\beta_\eta \colon B(\vec0,\delta)\rightarrow [0,\eta]$ 
such that $\beta_\eta\equiv \eta$ inside $B(\vec0,\delta/3)$ and $\beta_\eta\equiv 0$ outside $B(\vec0,\delta/2)$, and consider a $C^1$-function 
$\alpha_\eta: [0,\delta] \to [0,\eta]$
so that $\alpha_\eta(s)=0$ for every $s\in [0,\eta]\cup [\delta-\eta,\delta]$
and 
$\alpha_\eta(s)=\eta$ for every $s\in [2\eta,\delta-2\eta]$.
We define the $C^1$-vector field $Z\in \mathfrak X^1([0,\delta]\times B(\vec 0,\delta))$  by $Z=Z^0+P^\eta$ where $Z^0=(1, 0, \dots, 0)$ and 
$$
P^\eta(s,z)=\big( 0 \,,\,    \alpha_\eta(s) \cdot \beta_\eta(z) \cdot z\big).
$$ 
The vector fields $Z^0$ and $P^\eta$ commute,
$ 
\|Z-Z^0\|_{C^0} \le \eta^2
$  
and
$ 
\|Z_\delta \circ Z^0_{-\delta}(\delta,x)\|\le \eta^2 \|(\delta,x)\|
$
for every $x\in B(\vec 0, \delta/3)$.

\smallskip
Finally, let $Y=Y^3\in \mathfrak X^1(M)$ be the vector field defined by
$$ 
Y(x)= Dh_0(x)^{-1} Z(h_0(x))
$$ 
for each $x\in U_{0}$ and $Y=Y^2$ otherwise, 
%associated to the flow  $(Y^1_t)_{t\in %\mathbb R}$ 
%so that $Y$ is expressed by $Z$ in the flowbox
%coordinates $(s,z)$ in $U_{k-1}$
%given by 
%$$
%Y^1_t = h_{k-1}^{-1} \circ Z_t \circ h_{k-1}
%$$
%Moreover, using that the vector fields $Z^0$ and $P^\eta$ commute together with the Gronwall's inequality, one concludes that
%\begin{align*}
%d( Y^1_t(x), X_t(x)) 
 %   & = d( h_{k-1}^{-1} \circ Z_t \circ h_{k-1}(x), h_{k-1}^{-1} \circ Z^0_t \circ h_{k-1}(x)) \\
  %  & \le \|h_{k-1}^{-1}\|_{Lip} \cdot \; 
   % \| Z_t \circ h_{k-1}(x)-Z^0_t \circ %h_{k-1}(x)\| 
   % \\
   % & = \|h_{k-1}^{-1}\|_{Lip} \cdot \; 
   % \| P^\eta_t\circ Z^0_t \circ h_{k-1}(x)-Z^0_t \circ h_{k-1}(x)\| 
   % \\
   % & \le \|h_{k-1}^{-1}\|_{Lip} \cdot %e^{\eta^2 t}
%\end{align*}
%for every $x\in U_{k-1}$ and $t\in[0,%\delta]$.
and note that $ \|Y^3-Y^2\|_{C^0} \le C \eta^2$, where $C=\|Dh_0(\cdot)^{-1}\|_{C^0}>0$. Moreover, on 
$h_{0}^{-1}([0,\delta] \times {B(\vec0,\delta/3)}) \subset U_{0} $
the flow $\varphi^3$
is the composition of a translation and a strong transversal contraction. Altogether,  
$ 
f_{\varphi^3,I}^k(x_0)=x_0
$
and, as the hitting times $\bar \tau_{\varphi^3}$ and $\bar \tau_{\varphi^2}$ coincide,
\begin{align*}
f^k_{\varphi^3,I} 
    & = \varphi^{3}_{\bar \tau_{\varphi^3}} \circ I \circ f^{k-1}_{\varphi^2,I}    
     =  [ \varphi^3_{\delta} \circ \varphi^2_{-\delta} ] \circ f^k_{\varphi^2,I} 
\end{align*}
on $D\cap \overline{U_0}$.
Choosing $0<\eta \ll 1$ sufficiently small we obtain that $f^k_{\varphi^3,I}$ is a uniform contraction on an open neighborhood of $x_0$.   
%The vector field $Y=Y^3\in \mathfrak X^{1}(M)$ satisfies the requirements of the lemma.
By continuous dependence, the flow $\varphi^3$ generated by $Y^3$ is $\tau$-close to $\varphi$, since $Y^3$ is uniformly close to $X$.
\end{proof}

%%%%%%%%%%%%%%%%%%

\bigskip

%%%%%%%%
\subsection{Permanent periodic orbits}\label{Fix} 

Given a continuous map $f$ on a compact manifold $M$, 
we will recall the 
fixed point index used in \cite{CMN,H}. Let $B$ be an open ball on $M$ whose boundary $\partial B$ 
is  a embedded sphere and assume that either (i) $f(\overline {B}) \cap \overline {B}=\emptyset$, 
or (ii) $f(\overline B) \cup \overline B$ is contained in a single coordinate chart.
 If, in addition, $f$ has no fixed points in $\partial B$ then the \emph{fixed point index} $\iota_f( B)$ is defined as follows:
\begin{enumerate}
\item $\iota_f( B)=0$, if $f(\overline {B}) \cap \overline {B}=\emptyset$; and
\item $\iota_f( B)= \deg (\gamma)$, in the case that $f(\overline B) \cup \overline B$ is contained in a single 
		coordinate chart, where $\deg(\gamma)$ denotes the Hopf degree of the map 
		$\gamma : \partial B \simeq S^{n-1} \to S^{n-1}$ which is defined (after a change of coordinates) by
		$$
		\gamma(x)= \frac{f(x)-x}{\|f(x)-x\|}.
		$$
\end{enumerate}
%This notion is independent of the choice of local coordinates and it is locally constant in a small neighborhood of the continuous mappping $f$ (see e.g. \cite{Hi}). 

We will make use of the following properties of the fixed-point index for continuous maps, in the order we shall use them (cf. \cite{CMN}):
\begin{enumerate}
    \item[(P1)] 
if $f(B)\subset B$, $f\mid_B$ is a topological contraction and $p$ is the unique attracting fixed point in $B$ then 
$\iota_f(B)\neq 0$;
% THIS FOLLOWS BECAUSE ANY CONTRACTION IS HOMOTOPIC
%TO A LINEAR CONTRACTION AND, BY PAGE 131, EXERCISE 6
%IN \cite{GP}, AND THE FACT THAT THE INDEX IS PRESERVED
%BY HOMOTOPY, INDEX=(-1)^{\dim D}
    \item[(P2)]  if $\iota_f(B)$ is well defined then 
    $\iota_f(B)=\iota_g(B)$ 
    for every continuous map $g$ sufficiently $C^0$-close to $f$;
\item[(P3)]  if $\iota_f(B)\neq 0$ then $f$ has a fixed point in $B$.
\end{enumerate}

\medskip
In this framework of impulsive semiflows we will use the fixed-point index for Poincar\'e maps. 
More precisely, given a periodic orbit $\gamma$, a point $p\in \gamma$ and a local cross section $\Sigma$ to the impulsive semiflow passing through $p$,
the Poincar\'e return map is defined in an
open neighborhood $\Sigma_0\subset \Sigma$ of $p$ 
by
\begin{equation}
\label{defpoinc}
    P(x)
=
\begin{cases}
\begin{array}{cl}
   \varphi_{\tau^{+}_\Sigma(x)}(x)  & ,\, \text{if} \, \gamma \cap \hat D=\emptyset  \\
   (\varphi_{\tau^{+}_\Sigma(\cdot)} \circ P_{\varphi,I}^n \circ \varphi_{-\tau^-_{\hat{D}}(\cdot)})(x)  & ,\, \text{otherwise}, 
\end{array}
\end{cases}
\end{equation}
where $\tau^{\pm}_S: S \to \mathbb R^+$ denotes the first hitting time map of the flow $(\varphi_{\pm t})_{t\in \mathbb R}$ to a cross-section $S$ and $n=\#(\gamma \cap \hat D)$. Clearly, $P(p)=p$ if $\gamma\cap \hat D=\emptyset$ and $P^n(p)=p$, otherwise. Note that $P=P_{\varphi,I}\mid_{\Sigma_0}$.

%We will say that a periodic orbit $\gamma$ for the impulsive semiflow $\psi_{{\varphi},I}$ is an \emph{attracting periodic orbit} if there exists a local cross-section $\Sigma$ and $p\in \Sigma\cap \gamma$ that is an attracting fixed point for the Poincar\'e map $P$.

\iffalse
\begin{figure}[htb]
    \centering
%\includegraphics[width=0.8\linewidth]{index.png}
\includegraphics[width=0.67\linewidth]{periodicos.png}
    \caption{Three types of periodic orbits for the impulsive semiflows}
    \label{fig:enter-label}
\end{figure}
\fi

\color{black}

\medskip
We will use the following notions of permanence adapted from topological dynamics. 

\begin{definition}\label{def:permanence}
Fix 
%$X\in \mathfrak X^{0,1}(M)$ 
$\varphi \in \mathscr F^{0}(M)$
and $I\in \mathscr I_{D,\hat D}$ and a periodic orbit $\gamma$ for the impulsive semiflow $\psi_{{\varphi},I}$. We say that:

\begin{enumerate}
    \item[(a)] 
    $\gamma$ is \emph{permanent with respect to impulsive maps} (we write 
$\gamma\in \mathscr{P}_{impulsive}(\psi_{{\varphi},I})$)
 if for any impulsive map $J \in \mathscr I_{D,\hat D}$  $C^0$-arbitrarily close to $I$ the impulsive semiflow $\psi_{\varphi,J}$  has a 
periodic orbit near $\gamma$; 
    \item[(b)] 
    $\gamma$ is \emph{permanent with respect to 
flows} (we write simply $\gamma\in \mathscr{P}_{flows}(\psi_{{\varphi},I})
$ if for any flow
$\phi \in \mathscr F^{0}(M)$
that is $C^0$-arbitrarily close to $\varphi$ the impulsive semiflow $\psi_{\phi,I}$  has a periodic orbit near $\gamma$. 
\end{enumerate}
\end{definition}

For notational simplicity, we shall denote by $\mathscr{P}(\psi_{{\varphi},I})$ the set of all permanent closed orbits of $\psi_{{\varphi},I}$ irregardless, 
whenever there is no confusion on the perturbations one considers.

\smallskip
The following two results guarantees that, for typical impulsive semiflows, all periodic points are permanent. The first result concerns impulsive semiflows parameterized by elements in $\mathfrak{V}_{D,\hat D}$ . 
We denote by $\text{Per}_t(\psi_{{\varphi},I})$ the set of periodic points with period smaller or equal than $t$ of $\psi_{{\varphi},I}$.

\begin{proposition}\label{perm1}
Fix $I\in \mathscr I_{D,\hat D}$. There exists a $\tau$-residual subset $\mathscr{R}_I\subset \mathfrak{V}_{D,\hat D}$ such that 
$ 
\text{Per}(\psi_{{\varphi},I})=\mathscr{P}(\psi_{{\varphi},I}) \; \text{for any $\varphi \in \mathscr{R}_I$.}
$ 
\end{proposition}

\begin{proof} 
The proof exploits ideas developed in \cite{CMN}, adapted with the many subtleties of impulsive semiflows, due to possible discontinuities caused by impulses.
Fix $I\in \mathscr I_{D,\hat D}$
and let $({\mathcal{B}_i})_{i\ge 1}$ be a countable open basis of the topology on $M$, which is formed by balls whose boundaries are embedded spheres.
 For every $i, n \ge 1$, consider the pairwise disjoint subsets $\mathscr{F}_{i,n}, \mathscr{I}_{i,n}$  
of $\mathfrak{V}_{D,\hat D}$ defined as follows:
\begin{enumerate}
\item 
$\varphi \in \mathscr{F}_{i,n}$ if and only if  
$\overline{\mathcal{B}_i} \cap \bigcup_{0< t \le n} \text{Per}_t(\psi_{{\varphi},I}) =\emptyset$;
%$\gamma_t(x)\not=x$ for all $x\in\overline{\mathcal{B}_i}$ and all $t\in ]0,n]$;
\smallskip
\item $\varphi \in \mathscr{I}_{i,n}$ if and only if there exist $\overline{\mathcal{B}_j}\subset \mathcal{B}_i$, $\text{diam}(\mathcal{B}_j) < \text{diam}(\mathcal{B}_i),$ and a local cross-section $\Sigma$ so that
\begin{enumerate}
    %\smallskip
    %\item $\text{diam}(\mathcal{B}_j) < \text{diam}(\mathcal{B}_i),$ 
    \smallskip
\item 
$ 
\partial{(\mathcal{B}_j\cap \Sigma)} \cap \bigcup_{0< t \le n} \text{Per}_t(\psi_{{\varphi},I}) =\emptyset;
$ 
\smallskip
\item there exists $0< t < n$ so that $(\mathcal{B}_j \cap \Sigma) \cap \text{Per}_t(\psi_{{\varphi},I})\neq\emptyset$; 
\smallskip
\item 
$\iota_{P_{\varphi,I}}(\mathcal B_j \cap \Sigma)\neq 0$, where $P$ stands for the Poincar\'e map defined in \eqref{defpoinc} for the impulsive semiflow $\psi_{{\varphi},I}$.
\end{enumerate}
\end{enumerate}

Let us denote by $\mathring{\mathscr{F}}_{i,n}$ and $\mathring{\mathscr{I}}_{i,n}$ the $\tau$-interior of the subsets $\mathscr{F}_{i,n}$ and $\mathscr{I}_{i,n}$ of $\mathfrak{V}_{D,\hat D}$, respectively. We need the following.

\medskip
\noindent
{\bf Claim:} $\mathring{\mathscr{F}}_{i,n}\cup \mathring{\mathscr{I}}_{i,n}$ is a $\tau$-open and $\tau$-dense subset of $\mathfrak{V}_{D,\hat D}$
for each $i,n \ge 1$.
\medskip

\begin{proof}[Proof of the Claim]
 Fix $\varphi \in \mathfrak{V}_{D,\hat D}$.
By definition, $\varphi$ is $\tau$-approxima\-ted by some flow $\varphi^1$ generated by a Lipschitz vector field $X^1\in \mathfrak X^{0,1}(M)$ so that $X^1\pitchfork D$ and $X^1\pitchfork \hat D$. Now, it is enough to show that if
$\varphi^1\notin \mathring{\mathscr{F}}_{i,n}$ then it can be $\tau$-approxima\-ted by some flow in 
$\mathring{\mathscr{I}}_{i,n}$.

 Indeed, if $\varphi^1\notin \mathring{\mathscr{F}}_{i,n}$ then for every $\vep>0$ there exists 
$\varphi^2 \in \mathfrak{V}_{D,\hat D}$
 such that 
 $\varphi^2$ is $\tau$-close to $\varphi^1$ and
$\overline{\mathcal{B}_i} \cap \bigcup_{0< t \le n} \text{Per}_t(\psi_{\varphi^2,I}) \neq \emptyset$.
One can use Lemma~\ref{mario} to obtain a small $\tau$-perturbation $\varphi^3 \in  \mathcal F^{0,1}(M)$  of the flow $\varphi^2$, supported on a finite number of flowboxes which intersect $D$ and $\hat D$, in such a way that  there exists a periodic orbit intersecting the interior of ${\mathcal{B}_i}$ that does not intersect $\partial \hat D$.
By Lemma~\ref{pert1} one can perform a small $\tau$-perturbation $\varphi^4\in \mathcal F^{0,1}(M)$ of the flow $\varphi^3$ such that the impulsive flow $\psi_{\varphi^4,I}$ exhibits an attracting periodic orbit which intersects the interior of $\mathcal B_i$. 

We proceed to show that there exists an open subset $\mathcal B_j\subset \mathcal B_i$ so that items (2a)-(2c) above hold for every flow $\phi \in \mathcal F^{0,1}(M)$ that is $\tau$-arbitrarily close to $\varphi^4$ (which guarantees that $\varphi^4 \in \mathring{\mathscr{I}}_{i,n}$). Indeed, by property (P1) and the existence of the periodic attracting orbit, 
$\iota_{P_{\varphi^4,I}}(\mathcal B_j \cap \Sigma)\neq 0$.  
 Then the index
$\iota_{P_{\phi,I}}(\mathcal B_j \cap \Sigma)\neq 0$ for every $\phi$ that is $\tau$-close to $\varphi^4$, by property (P2).
Finally, the existence of a periodic orbit in $\mathcal B_j$ for impulsive semiflows $\psi_{\phi,I}$, with $\phi \in \mathcal F^{0,1}(M)$ that is sufficiently close to $\varphi^4$, is guaranteed by property (P3). 
%Since $\vep>0$ was chosen arbitrarily then the claim follows.
\end{proof}

\medskip
Finally we prove that the $C^0$-residual subset 
$$
\mathscr{R}_I:=\bigcap_{i,n \ge 1} [\mathring{\mathscr{F}}_{i,n}\cup \mathring{\mathscr{I}}_{i,n}] \subset \mathfrak{V}_{D,\hat D}
$$
satisfies the requirements of the proposition.
Take $\varphi \in\mathscr{R}_I$, a periodic orbit $\gamma\in \text{Per}(\psi_{{\varphi},I})$ of period $a>0$ and any $\mathcal{B}_i$ which intersects $\gamma$. 
As $\varphi \in \mathring{\mathscr{F}}_{i,n}\cup \mathring{\mathscr{I}}_{i,n}$ and $\gamma$ intersects $\mathcal B_i$ then there exist $n\ge 1$ and $\mathcal{B}_j$ 
with $\text{diam}(\mathcal{B}_j) < \text{diam}(\mathcal{B}_i)$ 
such that properties (2a)-(2c) hold. As the index is preserved by homotopy classes (recall property (P2)) then for any $\tau$-small perturbation $\phi$ of $\varphi$ one has
$\iota_{P_{\phi,I}}(\mathcal B_j \cap \Sigma)\neq 0$.
 Finally, property (P3) implies that 
$P_{\phi,I}$
 has a periodic orbit in $\mathcal B_j \cap \Sigma$, for every $\tau$-small perturbation of $\varphi$ (hence  $\psi_{\phi,I}$ has a periodic orbit close to $\gamma$ for every $\phi$ $\tau$-close to $\varphi$), which proves that $\gamma$ is permanent. 
This completes the proof of the proposition.
\end{proof}

\color{black}

The second result yields a similar consequence for impulsive semiflows parame\-terized by impulses.

\begin{proposition}\label{perm2}
Fix $\varphi\in \mathcal F^{0,1}(M)$ satisfying ~\eqref{eq:noncompact}. There exists a $C^0$-residual subset $\mathscr{R}$ of 
$\mathscr{I}_{D,\hat D}$ such that 
$$
\overline{Per(\psi_{{\varphi},I})} \cap \mathring{D} = 
\overline{\mathscr{P}(\psi_{{\varphi},I})}\cap \mathring{D}, \quad\text{for any $I \in \mathscr{R}$. }
$$
 \end{proposition}

\begin{proof}
\medskip
%This item is similar to the proof of \cite[Lemma ??]{BTV}.

The argument is similar to the one of Proposition~\ref{perm1}, with a decomposition formed on the space of impulsive maps.  For that reason we shall emphasize the main modifications.  Let $({\mathcal{B}_i})_{i\ge 1}$ be a countable open basis of the topology on $\mathring{D}$, which is formed by balls whose boundaries are embedded spheres and $\overline{{\mathcal B}_i}\subset\mathring{D}$
for every $i\ge 1$.
For every $i, n \ge 1$, consider the pairwise disjoint subsets $\mathscr{E}_{i,n}, \mathscr{H}_{i,n}$  
of $\mathscr{I}_{D,\hat D}$ defined as follows:
\begin{enumerate}
\item 
$I \in \mathscr{E}_{i,n}$ if and only if  
$\overline{\mathcal{B}_i} \cap \overline{\bigcup_{0< t \le n} 
\text{Per}_t(\psi_{{\varphi},I})} =\emptyset 
$;
%$\gamma_t(x)\not=x$ for all $x\in\overline{\mathcal{B}_i}$ and all $t\in ]0,n]$;
\smallskip
\item $I \in \mathscr{H}_{i,n}$ if and only if there exists $\overline{\mathcal{B}_j}\subset \mathcal{B}_i$, $\text{diam}(\mathcal{B}_j) < \text{diam}(\mathcal{B}_i)$ so that
\begin{enumerate}
    %\smallskip
    %\item $\text{diam}(\mathcal{B}_j) < \text{diam}(\mathcal{B}_i),$ 
    \smallskip
\item 
$ 
\partial{\mathcal{B}_j} \cap \overline{\bigcup_{0< t \le n} \text{Per}_t(\psi_{{\varphi},I})} =\emptyset;
$ 
\smallskip
\item there exists $0< t < n$ so that $\mathcal{B}_j \cap \overline{\text{Per}_t(\psi_{{\varphi},I})}\neq\emptyset$; 
\smallskip
\item 
$\iota_{f_{\varphi,I}}(\mathcal B_j)\neq 0$, where $f_{\varphi,I}$ is given by ~\eqref{eq:deffI}.
\end{enumerate}
\end{enumerate}

Let $\mathring{\mathscr{E}}_{i,n}$ and $\mathring{\mathscr{H}}_{i,n}$ be the $C^0$-interior of the subsets $\mathscr{E}_{i,n}$ and $\mathscr{H}_{i,n}$ of $\mathscr{I}_{D,\hat D}$, respectively. 
We claim that $\mathring{\mathscr{E}}_{i,n}\cup \mathring{\mathscr{H}}_{i,n}$ is a $C^0$-open and dense subset of $\mathscr{I}_{D,\hat D}$
for each $i,n \ge 1$.
Indeed, if $I\notin \mathring{\mathscr{E}}_{i,n}$ then it can be $C^0$-approximated by $I_0\in \mathscr{I}_{D,\hat D}$ so that $\overline{\mathcal{B}_i} \cap \overline{\bigcup_{0< t \le n} \text{Per}_t(\psi_{\varphi,I_0})} \neq \emptyset$.
Performing a $C^0$-arbitrary small perturbation $I_1\in \mathscr{I}_{D,\hat D}$ of $I_0$ one guarantees that $\psi_{\varphi,I_1}$ exhibits an attracting periodic orbit 
(cf. Lemma~\ref{cl2}) which intersects the interior of $\mathcal B_i$
(hence it is contained in $\mathring{D})$. 
Properties (P1)-(P3), applied to $f_{\varphi,I_1}$ ensure that  $I_1 \in \mathring{\mathscr{H}}_{i,n}$. 

Finally we are left to show that the 
$C^0$-residual subset 
$$
\mathscr{R}:=\bigcap_{i,n \ge 1} [\mathring{\mathscr{E}}_{i,n}\cup \mathring{ \mathscr{H}}_{i,n}] \subset \mathscr{I}_{D,\hat D}
$$
fulfills the requirements of the proposition. 
Since this proof
follows closely the argument used in the proof of the $C^0$-genericity of $\mathscr{R}_I$ (cf. proof of  Proposition~\ref{perm1}) we leave it as a simple exercise to the reader.
\end{proof}
\color{black}

\begin{remark}
In \cite{STV1} we proved that, given a smooth flow $(\varphi_t)_t$ for a $C^1$-generic impulsive map the periodic points whose closure intersects $D$ are all hyperbolic. This guarantees that a $C^0$-dense subset of impulsive maps have only permanent periodic orbits intersecting the closure of $D$. The similar argument could not be used in the proofs of Propositions~\ref{perm1} and ~\ref{perm2} above, as the flow is assumed to be only Lipschitz continuous and impulsive maps are only assumed to be homeomorphisms.   
\end{remark}

%%%%%%%%%%%%%%%%%%%%%%

\section{$C^0$-generic impulses}
\label{sec:proof1}

This section is devoted to the proof of Theorems~\ref{thmA}  on the properties of impulsive semiflows for $C^0$ typical impulses. 
Let $\varphi$ be a Lipschitz continuous flow generated by  $X\in \mathfrak{X}^{0,1}(M)$ and 
%{\color{magenta}{$D$ be a smooth
%submanifold of codimension one transversal to $X$
%satisfying ~\eqref{eq:noncompact}.}}
$D,\hat D\subset M$ be 
smooth codimension one submanifolds transversal to the flow 
such that hypothesis ~\eqref{eq:noncompact} holds.

We proceed to show that there exists a Baire generic subset 
$\mathscr R_\varphi
\subset \mathscr I_{D,\hat D}$
of impulsive maps
%$$
%\mathscr I_{D,\hat D}
	%=\big\{ I \in \text{Emb}^1(D, M) \colon %\hat D\cap{D} =\emptyset \;\text{and}\; \hat D %\pitchfork X \big\}
	%$$
such that every
$I \in \mathscr R_\varphi$ 
 satisfies
\begin{equation}
\label{eq:thmAeq21}
%\overline{Per_h(\psi_{{\varphi},I}) \cap \hat D} = %{\Omega(\psi_{{\varphi},I}) \cap \hat D 
\overline{Per(\psi_{{\varphi},I})} \cap \mathring{D} = {\Omega(\psi_{{\varphi},I}) \cap \mathring{D}}\,. %IR(\psi_{{\varphi},I}).
    \end{equation}

First, according to Proposition ~\ref{perm2},
there exists a $C^0$-residual subset $\mathscr{R}^0_\varphi$ of 
$\mathscr{I}_{D,\hat D}$ such that 
$ 
\overline{Per(\psi_{{\varphi},I})} \cap \mathring{D} = 
\overline{\mathscr{P}(\psi_{{\varphi},I})}\cap \mathring{D}, \text{for any $I \in \mathscr{R}^0_\varphi$. }
$
This guarantees that the map 
$$
\Gamma: \mathscr{I}_{D,\hat D} \to 2^M
\quad\text{given by}\quad
\Gamma(I)=\overline{\text{Per}(\psi_{{\varphi},I})} \cap \mathring{D}
$$
is lower semicontinuous in 
$\mathscr{R}^0_\varphi$. 
The continuity points of 
$\Gamma|_{\mathscr{R}^0_\varphi}$
form a $C^0$-residual subset 
$\mathscr{R}_\varphi\subset \mathscr{R}^0_\varphi$ (see \cite{Ku}).

\medskip
Fix $I\in \mathscr{R}_\varphi$.
We claim that \eqref{eq:thmAeq21} holds.  Assume otherwise, that there exists a point
$p\in [\Omega(\psi_{{\varphi},I})\cap \mathring{D}] \setminus [\overline{\text{Per}(\psi_{{\varphi},I})}\cap \mathring{D}]$.
%{\color{darkorange}{
%A $C^0$-small perturbation $I_1$ of $I$ guarantees 
%that $\psi_{\varphi,I_1}$ admits a periodic point $x$ close to $I_1(p)$ (recall Lemma~\ref{cl2}).  
%Under a small $C^0$-perturbation $I_2$ of $I_1$ 
%we can guarantee that $x$ is an attractor for the Poincar\'e map $P_{\varphi,I_2}: \hat D \to \hat D$ (recall ~\eqref{defPo00}) hence its orbit $\gamma$ is a permanent periodic orbit for $\psi_{\varphi,I_2}$. }}
A $C^0$-small perturbation $I_1$ of $I$ guarantees 
that $\psi_{\varphi,I_1}$ admits an attracting periodic point $x$ close to $I_1(p)$ (recall Lemma~\ref{cl2}). 
 Hence its orbit $\gamma$ is a permanent periodic orbit for $\psi_{\varphi,I_1}$. 

%{\color{darkorange}{
%By persistence of permanent periodic orbits and denseness of the $C^0$-Baire generic subset $\mathscr{R}_\varphi$ there exists $I_3\in \mathscr{R}_\varphi$ that is $C^0$-close to $I_2$ and so that $\psi_{\varphi,I_3}$ has a periodic orbit $q$ arbitrarily close to $p$. This is a contradiction with the fact that all impulsive maps in $\mathscr R_\varphi$ are continuity points of the map $\Gamma\mid_{\mathscr{R}^0_\varphi}$. 
%}}

By persistence of permanent periodic orbits and denseness of the $C^0$-Baire generic subset $\mathscr{R}_\varphi$ there exists $I_2\in \mathscr{R}_\varphi$ that is $C^0$-close to $I_1$ and so that $\psi_{\varphi,I_2}$ has a periodic orbit $q$ arbitrarily close to $p$. This is a contradiction with the fact that all impulsive maps in $\mathscr R_\varphi$ are continuity points of the map $\Gamma\mid_{\mathscr{R}^0_\varphi}$.

%%%%%%%%%%%%%%%%%%%
\section{$\tau$-generic flows}
\label{sec:proof2}

This section is devoted to the proof of Theorem~\ref{thmA-flows}. 
Throughout this section let  $D$ be a smooth submanifold of codimension one and
$I\in \mathscr I_{D,\hat D}$
be an impulsive map.
The proof of Theorem~\ref{thmA-flows}  builds over the fact that there exists a $\tau$-Baire generic subset 
$\mathscr{R}_I^0\subset \mathfrak{V}_{D,\hat D}$
such that 
$$
\text{Per}(\psi_{{\varphi},I})=\mathscr{P}(\psi_{{\varphi},I}), \quad \text{for any $\varphi \in \mathscr{R}_I^0$}
$$
which has been proved in Proposition~\ref{perm1}.
 This guarantees that the map $\Gamma$ defined by  $\Gamma(\varphi)=\overline{\text{Per}(\psi_{{\varphi},I})}$ for $\varphi \in \mathfrak{V}_{D,\hat D}$
 is lower semicontinuous in 
$\mathscr{R}_I^0$.
 In particular, the continuity points of 
$\Gamma|_{\mathscr{R}_I^0}$ 
 form a residual subset 
$\mathscr{R}_I\subset \mathscr{R}_I^0$.

\medskip
We claim that $\Omega(\psi_{{\varphi},I})=\overline{\text{Per}(\psi_{{\varphi},I})}$ for every 
$\varphi \in \mathscr{R}_I$.
Indeed, if this was not the case then choose 
$p\in \Omega(\psi_{{\varphi},I}) \setminus \overline{\text{Per}(\psi_{{\varphi},I})}$.
Given an $\varepsilon>0$ there exists a $\tau$-$\varepsilon$-perturbation  $\varphi^{\vep} \in \mathcal{F}^{0,1}(M)$  of $\varphi$, generated by a vector field $Y_{\vep} \in \mathfrak{X}^{0,1}(M)$ so that $Y_{\vep}\pitchfork D$ and   $Y_{\vep} \pitchfork \hat{D}$, a point $q_{\vep} \in M$ and $t_{\vep}>0$ so that $d(p, q_\vep) < \vep$, $d(p, \varphi^{\vep} (t_\vep, q_\vep) ) < \vep$ and $d(\varphi^{\vep} (t_\vep, q_\vep) , q_\vep) < \vep$.

The transversality of a given vector field to the compact cross sections $\bar D$ and $\bar{\hat D}$ is a $C^0$-open condition. This, combined with Lemma~\ref{cl2vectors}, implies that there exists a $C^0$-small perturbation $Y_1\in \mathfrak{V}_{D,\hat D}$ of $X$ such that $p$ is a periodic orbit for $\psi_{\varphi^{1},I}$
(we will denote by $\varphi^{i}$ the Lipschitz flow generated by $Y_i$). Under a small $C^0$-perturbation $Y_2$ of $Y_1$ we can guarantee that $p$ is a permanent periodic orbit for 
$\psi_{\varphi^{2},I}$ (recall Lemma~\ref{pert1}). By Gronwall's inequality the flow $\varphi_2$ is $\tau$-close to $\varphi$. Moreover, by persistence of permanent periodic orbits and denseness of the $\tau$- generic subset 
$\mathscr{R}_I$,
we can perform a $C^0$-small perturbation $Y_3$ of $Y_0$ in such a way that 
$\psi_{\varphi^{3},I}$
has a periodic orbit $q$ arbitrarily close to $p$. This is a contradiction with the fact that 
all flows in
$\mathscr{R}_I$
are continuity points of the map 
$\Gamma\mid_{\mathscr{R}_I^0}$.
This finishes the proof.

%%%%%%%%%%%%%%%%%%%%%
\section{Examples}\label{sec:examples}

In this section we shall provide  simple  examples of the main results in order to illustrate possible applications. The first application concerns the class of impulsive Anosov flows.

\begin{example} ($C^0$-generic impulsive Anosov flows)
    Let $X$ be a $C^1$-vector field on a compact three-dimensional compact manifold generating a transitive Anosov flow $\varphi$, let 
$D, \hat D$ be codimension one smooth submanifolds of $M$ transversal to $X$ and $I\in \mathscr{I}_{D,\hat D}$.

    %$C^1$ codimension one Anosov flow $\phi$ on a  compact  manifold $M$ which admits a global cross-section, i.e. an Anosov flow which admits a smooth codimension one submanifold $D$
    %(see e.g. \cite{Ghys,Plante} for several partial classifications of such classes of flows).
Let 
$\mathcal U \subset  \mathscr{F}^{0,1}(M)$
be a small open neighborhood of 
$\varphi$
(in the $C^0$ topology) 
%\textcolor{blue}{Attention! vector field in $\mathcal U$ could be only $C^0$ and do not generate a dynamical system! (multivalued flow instead...)}) 
such that 
$\mathcal U \subset \mathfrak{V}_{D,\hat D}$.
Even though it is not known whether all flows 
%{\color{blue}{\sout{generated by vector fields}}}
 in $\mathcal U$ are equivalent, hence have a dense set of periodic orbits, by the proof of Theorem~\ref{thmA-flows}, 
periodic points are all permanent and 
the set of periodic points is dense in the non-wandering set, for a $C^0$-generic choice of 
%{\color{magenta}{$Y\in \mathcal U$.}}
$\phi\in \mathcal U$.

Furthermore, Theorem~\ref{thmA} implies that the set of (permanent) periodic orbits is dense in the impulsive non-wandering set for a $C^0$-generic choice of the impulses 
$I\in \mathscr{I}_{D,\hat D}$. 
\end{example}

\begin{example}(Perturbations of minimal flows) \label{ex:minimal}
Let $\varphi$ be a minimal flow on $\mathbb T^3$ generated by a $C^1$ vector field $X\in \mathfrak{X}^{1}(\mathbb T^3)$ and 
$D,\hat D\subset M$ be 
smooth codimension one 
closed
submanifolds transversal to the flow.
 The flow $\varphi$ has no periodic orbits. Nevertheless, Theorem~\ref{thmA} ensures that for a $C^0$-generic choice of the impulse 
$I\in \mathscr I_{D,\hat D}$,
  the impulsive semiflow $\psi_{{\varphi},I}$ has a dense set of periodic orbits in the impulsive non-wandering set. 
  Moreover, Theorem~\ref{thmA-flows} ensures that permanent periodic orbits are dense in the whole non-wandering set.

\end{example}

\medskip

Finally, we finish this section by considering a class of impulsive Lorenz attractors, and prove that even though the impulsive semiflow may not 
exhibit partial hyperbolicity one can still prove that for a typical impulsive Lorenz attractor 
the periodic orbits are dense 
in the non-wandering set.

\begin{example}(Impulsive flows derived from geometric Lorenz attractors)\label{ex:Lorenz}
%The geometric Lorenz attractors were introduced %independently by Guckenheimer and Williams~\cite{gw} 
%and Afra$\breve {\rm \i}$movi$\check{\rm c}$, Bykov and %Sil'nikov, \cite{abs}
%for vector fields on a closed 3-manifold $M^3$. 
	%We say that 
 A vector field $X\in\mathscr{X}^r(M^3)$ ($r\geq1$) exhibits a \emph{geometric Lorenz attractor} $\Lambda$
 if $X$ has an attracting region $U\subset M^3$ such that  $\Lambda=\bigcap_{t>0}\phi^X_t(U)$ is a singular hyperbolic attractor and satisfies:
	\begin{itemize}
		\item[(i)] There exists a unique singularity $\sigma\in\Lambda$ with three exponents $\lambda_1<\lambda_2<0<\lambda_3$, which satisfy $\lambda_1+\lambda_3<0$ and $\lambda_2+\lambda_3>0$, whose eigenspaces (in local coordinates $x_1, x_2, x_3$ in $\mathbb R^3$) are identified with the canonical axis.  
		\item[(ii)] $\Lambda$ admits a $C^r$-smooth cross section which in local coordinates can be written as
		$\Sigma=[-1,1]  \times [-1,1]  \times \{1\}$ and
		for every $z\in U\setminus W^s_{\it loc}(\sigma)$, there exists $t>0$ such that $\phi_t^X(z)\in\Sigma$
		(here $W^s_{\it loc}(\sigma)$ stands for the stable manifold of the hyperbolic singularity $\sigma$)
		\item[(iii)] With the identification
		of $\Sigma$ with $[-1,1]\times[-1,1]$ (by a $C^1$-diffeomorphism) where $l=\{0\}\times[-1,1]=W^s_{\it loc}(\sigma)\cap\Sigma$, 
		the Poincar\'e map $P:\Sigma\setminus l\rightarrow\Sigma$ is a skew-product map
		$$
		P(x,y)=\big( f(x)~,~H(x,y) \big), \qquad \forall(x,y)\in[-1,1]^2\setminus l
		$$
		where
		\begin{itemize}
			\item $H(x,y)<0$ for $x>0$, and $H(x,y)>0$ for $x<0$; 
			\item $\sup_{(x,y)\in\Sigma\setminus l}\big|\partial H(x,y)/\partial y\big|<1$, 
			and
			$\sup_{(x,y)\in\Sigma\setminus l}\big|\partial H(x,y)/\partial x\big|<1$; 
			\item the one-dimensional quotient map $f:[-1,1]\setminus\{0\}\rightarrow[-1,1]$ is $C^1$-smooth and satisfies
			$\lim_{x\rightarrow0^-}f(x)=1$,
			$\lim_{x\rightarrow0^+}f(x)=-1$, $-1<f(x)<1$ and
			$f'(x)>\sqrt{2}$ for every $x\in[-1,1]\setminus\{0\}$.
		\end{itemize}
	\end{itemize}

%The set of vector fields exhibiting a geometric Lorenz attractor forms a $C^1$-open set in $\mathscr{X}^r(M)$ (see e.g.~\cite[Proposition 4.7]{STW}). Moreover, each geometric Lorenz attractor is a singular-hyperbolic homoclinic class, the map $P$ has a dominated splitting and the flow has a dense set of periodic orbits
%(cf. ~\cite{AP} for definitions and proofs).

\smallskip
The Poincar\'e map $P$ of the geometric Lorenz attractors can be written as a composition $P=P_2\circ P_1$, where 
$P_1: D \to \hat D$ and $P_2: \hat D \to D$ are Dulac maps, i.e.
the first hitting time maps for the flow $(\phi^X_t)$ between $D:=\Sigma\setminus l$ and $\hat D=\{\phi^X_{h(x)}(x) \colon x \in D\}$, where 
$h(x)=\inf\{t>0 \colon \phi^X_{h(x)}(x) \in \{|x_1|=1\}\}$.
%where the function $\tau_1$ is piecewise $C^1$.

Theorems~\ref{thmA} and~\ref{thmA-flows} guarantee that there exist permanent periodic orbits under 
generic perturbations 
of $P_1$
and the flow, respectively.
\end{example}

%%%%%%%%%%%%%%%%%%%%%%%%%%%%%%%%%%%%%%%%%%%%%%%%%%%%%
\section{Final comments}
\label{sec:finalc}

\subsection*{Volume preserving impulsive semiflows}

A particular interesting framework, with many physical motivations, concerns the case of volume preserving dynamical systems. One can  define a volume preserving impulsive semiflow 
as one generated by a volume preserving flow and an impulsive map $I: D\to \hat D$ such that $\text{Leb}_{\hat D}(\cdot)=\text{Leb}_{D}(I^{-1}(\cdot))$, where $\text{Leb}_{D}$ and $\text{Leb}_{\hat D}$ denote the induced volume measures on the submanifolds $D,\hat D$ of $M$. 
Assume that 
$\varphi$
is a volume preserving Lipschitz continuous  flow on a smooth compact manifold of dimension larger or equal than $3$ and that $D,\hat D$ are smooth cross-sections transversal to the flow.  
In particular, in view of \cite{OU} and the main results of this paper one expects the following questions to have an affirmative answer:

\medskip
\noindent
{\bf Question 1.} If the flow 
$\varphi$
is ergodic does a $C^0$-generic impulse $I$ generate an ergodic impulsive semiflow $\psi_{{\varphi},I}$? Moreover, if $I :D \to \hat D$ preserves volume does a $C^0$-generic flow $ \varphi \in \mathfrak V_{D,\hat D}$ generate an ergodic impulsive semiflow $\psi_{{\varphi},I}$?
\medskip

%\noindent
%{\bf Question 2.} There exist examples of non-ergodic volume preserving flows  $(X_t)_t$ so that, for a $C^0$-generic impulse $I$, the  impulsive semiflow $\psi_{{\varphi},I}$ is ergodic?
%\medskip

We also expect that the main results of this paper (Theorems~\ref{thmA} and ~\ref{thmA-flows}) may find a parallel for $C^0$-typical volume preserving impulsive semiflows.

%As the answer to the previous questions would involve special perturbation theory within the classes of volume preserving maps and flows, in order for the resultant impulsive semiflow to preserve volume, these do not fit in the general description aimed in this paper. 

\subsection*{Shadowing and gluing orbit properties}
It is known that $C^0$-generic homeomorphisms satisfy the shadowing property (see e.g. \cite{Akin3} and references therein). The latter, in presence of the denseness of periodic orbits, implies on the gluing orbit property, which itself ensures a strong transitivity, denseness of ergodic measures on the simplex of invariant measures, among several strong recurrence results (cf. \cite{BTV,BoTV}).  Let us recall some ingredients.
Recall that, given real numbers $\delta>0$ and $T > 1$, we say that a pair of sequences $[x_i, t_i]_{i\in \mathbb{Z}}$,  where $x_i \in M$, $t_i \in \R$, $1\leq t_i \leq T$, 
is a $(\delta,T)$-{\em pseudo-orbit} of  $\psi_{{\varphi},I}$ if 
$$
d(\gamma_{x_i}(t_i),x_{i+1}) < \delta \,\,\text{for all }\,\,i \in \mathbb{Z}.
$$
%We will say for simplicity that such a pseudo-orbit has $j$-jumps.
For the sequence $(t_i)_{i \in \Z}$ we write,
$\sigma(n)=t_0+t_1+\ldots+t_{n-1}$ if $n>0$, $\sigma(n)=-(t_n+\ldots+t_{-2}+t_{-1})$ if $n<0$ and $\sigma(0)=0$. 
Let $x_0 \star t$ denote a point on a $(\delta,T)$-chain $t$ units time from $x_0$. More precisely, for $t \in \R$,
$x_0 \star t=\gamma_{t-\sigma(i)}(x_i) \hspace{0.4cm} {\text{\rm if}} \hspace{0.4cm} \sigma(i) \leq t <\sigma(i+1).$
By $\mbox{Rep}$ we denote the set of all increasing homeomorphisms $\zeta\colon \R \rightarrow \R$,
called {\em reparametrizations}, satisfying $\zeta(0)=0$. Fixing $\varepsilon>0$, we define the set
$$
\mbox{Rep}(\varepsilon)=\left\{\zeta \in \mbox{Rep}: \left|\frac{\zeta(t) -\zeta(s)}{t-s}-1 \right|<\varepsilon, \, s, t \in \R \right\},
$$ 
of the reparametrizations $\varepsilon$-close to the identity.
%In rought terms, a reparametrization $\zeta : \mathbb R \to \mathbb R$ belongs to $\mbox{Rep}(\varepsilon)$ whenever 
%its velocity at all points is $\vep$-close to 1.
%
 The impulsive semiflow $\psi_{{\varphi},I}$ satisfies the
\emph{shadowing property} %on $\Lambda$ 
if, for any $\varepsilon>0$ and $T > 1$ there exists $\delta=\delta(\varepsilon, T)>0$ such that for any 
 $(\delta,T)$-pseudo-orbit  $[x_i, t_i]_{i \in \Z}$ there is %$\tilde x \in \Lambda$ 
 $\tilde x \in M$ and a reparametrization $\zeta \in \mbox{Rep}(\varepsilon)$ such that 
$
d(\gamma_{\tilde x}(\zeta(t)),x_0 \star t)<\varepsilon,\text{ for every } t \in \R.
$ 

%\medskip
%\noindent
%{\bf Question 4.} Assuming that $(X_t)_t$ is  ergodic,  does a $C^0$-generic impulse $I$ generate an  impulsive semiflow $\psi_{{\varphi},I}$ satisfying the shadowing and specification properties?
%\medskip

\medskip
While the shadowing property turns out to be $C^0$-generic in the space of continu\-ous flows generated by Lipschitz continuous vector fields \cite{BTV} this is no longer true for impulsive semiflows, as shown in the next example.

\begin{example}
    %if $X$ is a gradient vector field on $\mathbb S^2$ then $\varphi$ is 
Assume that the vector field $X\in \mathfrak{X}^1(\mathbb S^2)$ generates the usual north-pole south-pole flow 
on the sphere $\mathbb S^2$, that $D\subset \mathbb S^2$ denotes the meridian, $\hat D=\varphi_{-1}(D)$ and $\tau(x)=\sup\{|t| \colon\varphi_t(x) \notin D\}$. Then, given a small $\vep>0$, for every vector field $Y\in \mathfrak{X}^{0,1}(\mathbb S^2)$ that is $C^0$-close to $X$, any impulsive map $I\in \mathscr{I}_{D,\hat D}$ and $\delta>0$, it is simple to check that the following $\delta$-pseudo-orbit $[x_i,t_i]_{i=1,2}$ is not $\vep$-shadowed: it is enough to take any regular point $x_1$ in the interior of the upper hemisphere, $t_1=\tau(x_1)-\frac\delta{10 \|X\|_\infty}$, $x_2=\varphi_{\tau(x_1)+\frac\delta{10 \|X\|_\infty}}(x_1)$ and $t_2=1$. The absence of shadowing holds by $C^0$-perturbations of the Lipschitz vector field and impulsive map $I$.
\end{example}

The previous simple example clarifies that, in order for some shadowing property to hold for typical impulsive semiflows, it may be necessary to consider special classes of pseudo-orbits.
This justifies the following:

\medskip
\noindent
{\bf Question 2.} Does a typical impulsive semiflow (parameterized by either the space of impulsive maps or flows) satisfy a shadowing property for some relevant class of pseudo-orbits?

\bigskip

%%%%%%%%%%%%%%%%%%%%%%%%%%%%%%%%
{\bf Acknowledgements.} 
Part of this work was done during the visit of the third named author to UFRJ and during the visit of the fourth named author to UFRJ and PUC-Rio, whose research environment and research conditions are acknowledged.

MB was partially financed by Portuguese Funds through FCT (Funda\c c\~ao para a Ci\^encia e a Tecnologia) through the research Project
 UID/04561/2025 (https:// doi.org/10.54499/UID/04561/2025).

JS was funded by
the grant E-26/010/002610/2019, Rio de Janeiro Research Foundation (FAPERJ), and by the Coordena\c c\~ao de Aperfei\-çoamento de Pessoal de N\'ivel Superior - Brasil (CAPES), Finance Code 001.

MJT was partially financed by Portuguese Funds through FCT (Funda\c c\~ao para a Ci\^encia e a Tecnologia) through the research Project
 UID/00013/2025 (https:// doi.org/10.54499/UID/00013/2025).

PV was partially supported by CIDMA, through FCT (Funda\c c\~ao para a Ci\^encia e a Tecnologia), I.P., under the projects with references 
 UID/04106/2025 (https://doi.org/\-10.54499/UID/04106/2025)
 and UID/PRR/04106/2025
(https://doi.org/\-10.54499/\-UID/\-PRR/04106/2025).

\end{document}